\documentclass{article}

\usepackage[utf8]{inputenc}
\usepackage[english]{babel}

\usepackage{graphicx,epstopdf}
\usepackage{pgfplots}
\pgfplotsset{compat=1.17}
\usetikzlibrary{pgfplots.groupplots}
\usepackage[caption=false]{subfig}

\usepackage{amsfonts}
\usepackage{amsmath}
\numberwithin{equation}{section}
\usepackage{amssymb}
\usepackage{mathtools}
\usepackage[section]{placeins} 
\usepackage{calc} 
\usepackage{xcolor}
\usepackage{amsthm} 
\theoremstyle{definition}

\usepackage{algorithm,algpseudocode}

\usepackage{circuitikz}
\usepackage{geometry}
\usepackage{mathtools} 
\usepackage{setspace}
\usepackage{hyperref} 
\usepackage[capitalise]{cleveref}
\Crefname{equation}{eq.}{eqs.}
\Crefname{ALC@unique}{Line}{Lines}
\Crefname{line}{line}{lines}
\newtheorem{definition}{Definition}[section] 
\theoremstyle{plain}
\newtheorem{theorem}{Theorem}[section]
\newtheorem{proposition}{Proposition}[section]
\newtheorem{corollary}{Corollary}[theorem]
\newtheorem{lemma}[theorem]{Lemma}
\theoremstyle{remark}
\newtheorem{remark}{Remark}
\usepackage{bm}

\newcommand{\eps}{\varepsilon}

\newcommand{\abs}[1]{\vert {#1} \vert}

\newcommand{\rank}{\operatorname{rank}}

\newcommand{\ip}[2]{\left\langle {#1},\ {#2} \right\rangle}

\newcommand{\norm}[1]{\Vert {#1} \Vert}
\newcommand{\normsq}[1]{\norm{#1}^2}

\newcommand{\enorm}[1]{\norm{#1}_2}

\newcommand{\fnorm}[1]{\norm{#1}_{\mathrm{F}}}

\newcommand{\bigO}{\mathrm{O}}

\newcommand{\Rn}[1]{\mathbb{R}^{#1}}
\newcommand{\Cn}[1]{\mathbb{C}^{#1}}
\newcommand{\Rnm}[2]{\mathbb{R}^{#1 \times #2}}

\renewcommand{\Cn}[1]{\mathbb{C}^{#1}}

\newcommand{\Om}{\Omega} 
\newcommand{\Omperp}{\perp^{\Om}}
\newcommand{\Omperpornot}{\perp^{(\Om)}}

\newcommand{\K}{\mathcal{K}} 
\newcommand{\Lspace}{\mathcal{L}}

\newcommand{\POm}{P_{\K}^{\Om}}

\newcommand{\A}{A} 

\newcommand{\Hsa}{H_\sa} 
\newcommand{\Hba}{H_{\ba}} 
 
\newcommand{\tHba}{\tilde{H}_{\ba}} 

\newcommand{\Bsa}{B_\sa} 
\newcommand{\Bba}{B_{\ba}} 
\newcommand{\tBsa}{\tilde{B}_\sa} 

\newcommand{\Tsa}{T_\sa} 
\newcommand{\Tba}{T_{\ba}} 
\newcommand{\tTba}{\tilde{T}_{\ba}} 
\newcommand{\tTsa}{\tilde{T}_\sa} 

\newcommand{\V}{V_{\sa}} 
\newcommand{\Vba}{V_{\ba}} 
 
\newcommand{\tVba}{\tilde{V}_{\ba}} 

\newcommand{\U}{U_{\sa}} 
\newcommand{\Uba}{U_{\ba}} 
\newcommand{\tU}{\tilde{U}_{\sa}} 
\newcommand{\tUba}{\tilde{U}_{\ba}} 

\newcommand{\Ssa}{S_{\sa}}
\newcommand{\Sba}{S_{\ba}}
\newcommand{\tSsa}{\tilde{S}_{\sa}}

\newcommand{\slast}{s_{\sa+1}}
\newcommand{\sba}{s_{\ba+1}}
\newcommand{\tslast}{\tilde{s}_{\sa+1}}

\newcommand{\eivec}{x} 
\newcommand{\eival}{\lambda} 
\newcommand{\rivec}{\tilde{\eivec}} 
\newcommand{\yrivec}{{y}} 
\newcommand{\rival}{\tilde{\eival}} 

\newcommand{\uinit}{u_1} 
\newcommand{\uinitj}{u_1^+} 

\newcommand{\ulast}{u_{\sa+1}} 
\newcommand{\uba}{u_{\ba+1}} 
\newcommand{\tulast}{\tilde{u}_{\sa+1}} 

\newcommand{\blast}{b_{\sa}} 
\newcommand{\bba}{b_{\ba}} 
\newcommand{\tblast}{\tilde{b}_{\sa}} 
\newcommand{\tbba}{\tilde{b}_{\ba}}

\newcommand{\tsp}{^T}

\newcommand{\sa}{k} 
\newcommand{\nshi}{p} 
\newcommand{\ba}{\sa + \nshi} 
\newcommand{\dsk}{d}

\renewcommand{\epsilon}{\varepsilon}

\newcommand{\rKS}{randomized Krylov-Schur}
\newcommand{\rArno}{randomized Arnoldi}

\DeclareMathOperator{\Sp}{span}
\DeclareMathOperator{\diag}{diag}
\newcommand{\Span}[1]{\Sp \{ #1 \} }
\newcommand{\ndim}{n}

\newenvironment{keywords}
 {\par\textbf{Keywords.}\ \ignorespaces}
 {\par\medskip}
\newenvironment{AMS}
 {\par\textbf{AMS subject classification.}\ \ignorespaces}
 {\par\medskip}

\title{Randomized Krylov-Schur eigensolver with deflation}

\author{Jean-Guillaume de Damas\thanks{Sorbonne Universite, Inria, Universite de Paris Laboratoire Jacques-Louis Lions, Paris, France}, Laura Grigori\thanks{Laboratory for Simulation and Modelling, Paul Scherrer Institute, Switzerland; Institute of Mathematics, EPFL, Switzerland}}
\begin{document}

\maketitle

\begin{abstract}
   This work introduces a novel algorithm to solve large-scale eigenvalue problems and seek a small set of eigenpairs. The method, called randomized Krylov-Schur (rKS), has a simple implementation and benefits from fast and efficient operations in low-dimensional spaces, such as sketch-orthogonalization processes and stable reordering of Schur factorizations. It also includes a practical deflation technique for converged eigenpairs, enabling the computation of the eigenspace associated with a given part of the spectrum. Numerical experiments are provided to demonstrate the scalability and accuracy of the method.
\end{abstract}



\section{Introduction}
The focus of this work is to solve the eigenvalue problem $\A \eivec_i = \eival_i  \eivec_i$ for a given matrix $\A \in \Rnm{n}{n}$ and a small subset of eigenpairs indexed by $i = 1,\dots,\sa$. The eigenvalues $\eival_i$ and eigenvectors $\eivec_i$ provide critical insights into the properties of a wide range of systems and phenomena. 
In physics, eigenvalues describe natural frequencies in mechanical vibrations and quantum systems. In engineering, they help optimize designs and analyze the stability of structures and control systems. Eigenvectors are fundamental in data science for techniques such as Principal Component Analysis (PCA) that allows reducing dimensionality and identifying patterns in large datasets. Several domain decomposition methods also rely on solving eigenvalue problems in subdomains to accelerate the solving of linear systems of equations arising in areas such as fluid dynamics and electromagnetic. See \cite{Pain2005PhysicsVibrationsWaves,Lin2020stateartreview,Jolliffe2016Principalcomponentanalysis,Nataf2024GenEODomainDecomposition} for these applications. The eigenvalue problem has been extensively studied in numerical linear algebra over the past decades; see e.g. \cite{Wilkinson1988algebraiceigenvalueproblem,saad2011numerical,Kressner2005NumericalMethodsGeneral,Golub2000Eigenvaluecomputation20th}.  In this paper, we consider in particular the restarted Krylov framework, notably the Implicitly Restarted Arnoldi method (IRA) derived in \cite{Sorensen1992ImplicitApplicationPolynomial} and improved in \cite{Lehoucq1995DeflationTechniquesImplicitly}, and the subsequent Krylov-Schur method (KS) presented in \cite{Stewart2002KrylovSchurAlgorithm}. Krylov methods are based on the Krylov subspace
$$\K_\sa(\A, v_1) \coloneqq \Span{v_1, \A v_1, \A^2 v_1, \dots, \A^{\sa-1} v_1}$$
for $v_1 \in \Rn{n}$ and are widely used in numerical linear algebra to solve linear systems of equations and eigenvalue problems; see e.g. \cite{Saad2003Iterativemethodssparse,saad2011numerical} for the two problems, respectively.

The use of randomization in numerical linear algebra has enabled many improvements in the efficiency, accuracy, or scalability of algorithms for matrix factorization, eigenvalue computation, and linear system solving, especially for large-scale datasets. Also called sketching, randomization enables the embedding of subspaces in high dimensional vector space, such as a basis for $\K(\A,v_1)$ consisting of $\sa$ vectors in the large space $\Rn{n}$, into a much smaller vector space $\Rn{\dsk}$ with $\dsk \ll n$, thereby reducing computational and communication costs. This technique has been explored for different problems in numerical linear algebra as matrix multiplication \cite{Drineas2006FastMonteCarlo}, low-rank matrix approximation \cite{Drineas2006FastMonteCarloa,HalkoFindingStructureRandomness2011}, or solving least-squares problems \cite{Rokhlinfastrandomizedalgorithm2008}. An extensive overview of randomized linear algebra can be found in e.g. \cite{Boyd2010RandomizedAlgorithmsMatrices,WoodruffComputationalAdvertisingTechniques2014, MartinssonRandomizednumericallinear2020, Murray2022RandomizedNumericalLinear}. Sketching has been applied to Krylov methods in different works as \cite{BalabanovRandomizedGramSchmidt2022,Nakatsukasa2024FastAccurateRandomized,Timsit2023RandomizedOrthogonalProjection,Guettel2023sketchselectArnoldi,Guettel2023RandomizedSketchingKrylov}, mainly to obtain well-conditioned bases of the Krylov subspace, which leads to efficient randomized variants of the GMRES algorithm for linear systems of equations. Restarting strategies in this context are developed further in \cite{Burke2023GMRESrandomizedsketching,Jang2024RandomizedflexibleGMRES}. For the nonsymmetric, non-generalized eigenvalue problem considered here, possible sketchings of the Rayleigh-Ritz method are introduced in \cite{Balabanov2021RandomizedblockGram,Nakatsukasa2024FastAccurateRandomized}, and a randomization of IRA is derived in \cite{Damas2024RandomizedImplicitlyRestarted}. In \cite{Simoncini2025StabilizedKrylovSubspace}, the authors showed that the more general concept of Krylov decomposition can encompass randomized factorizations, and they notably develop their theory in the context of truncated orthogonalization procedures. In other contexts, such as nonlinear or generalized eigenvalue problems, randomized eigensolvers have been introduced in \cite{GuettelRandomizedsketchingnonlinear2022} and \cite{Saibaba2015Randomizedalgorithmsgeneralized,Kalantzis2021FastRandomizedNon,Yin2016randomizedFEASTalgorithm}, respectively.

In this article, we present a new method for solving large scale eigenvalue problems when few eigenpairs are sought. We call it randomized Krylov-Schur (rKS), and it can be seen as an extension of the randomized Implicitly Restarted Arnoldi algorithm (rIRA) from \cite{Damas2024RandomizedImplicitlyRestarted}. In particular, we derive an equivalence between the two methods in \Cref{th:KDeqRA}. The name of the method reflects its connection to the Krylov-Schur (KS) algorithm of \cite{Stewart2002KrylovSchurAlgorithm} since rKS shares two important concepts with it, namely Krylov factorization and Schur decomposition. For rKS, we define here novel randomized factorizations called sketch-orthonormal Krylov(-Schur) decompositions that extend the randomized Arnoldi factorization, and we discuss how they can be obtained. We derive the rKS method in \cref{alg:rKS} that is based on these factorizations, together with an integration of the restart strategy developed in \cite{Stewart2002KrylovSchurAlgorithm}. Compared to other deterministic eigensolvers, we show that rKS has a significant speedup in computation time because of the use of sketch-orthogonalization processes, allowing better scalability for large-scale applications. Compared to rIRA, it enables not only an easier implementation, but also a more stable one due to the use of efficient methods to reorder the Schur factorization, rather than using the unstable implicit shifted QR method, as mentioned in \cite{Lehoucq1995DeflationTechniquesImplicitly}. Thanks to the framework of sketch-orthonormal Krylov-Schur decomposition, we define a simple deflation technique for converged eigenpairs. This makes rKS capable of computing the subspace associated with a given set of eigenvectors in an efficient manner, and we provide bounds on the quality of this approximate invariant subspace.

The paper is organized as follows. \Cref{sec:prelim} introduces notations and randomization techniques. In \cref{sec:rfactorizations}, we derive different factorizations that generalize the randomized Arnoldi framework, while still allowing for an efficient sketch Rayleigh-Ritz procedure based on Krylov subspaces. One of the advantages of the newly introduced \rKS{} factorization is that it enables a restarting procedure based on reordering Schur forms of matrices, resulting in the central \rKS{} eigensolver derived in \cref{sec:rKS}. In \cref{sec:deflation}, we prove how deflation can be incorporated with respect to sketch-orthonormal procedures. Lastly, the numerical efficiency of the method is tested in \cref{sec:numericals}.

\subsection{Preliminaries}
\label{sec:prelim}
In this work, the Euclidean vector space $\Rn{n}$ is equipped with the canonical inner product $\ip{x}{y} \coloneqq \sum_{i=1}^{n} x_i y_i = x^Ty$, where $x_i$ and $y_i$ are entries of vectors $x$ and $y$. This induces the $2$-norm on vectors, $\norm{x} \coloneqq \sqrt{\ip{x}{x}}$. Matrices are denoted with capital letters, and the notation $\norm{M}$ can refer to any sub-multiplicative norm. When precision is needed, an index such as $\fnorm{M}$ or $\enorm{M} \coloneqq \max_{x \in \Rn{n}} \frac{\norm{Ax}}{\norm{x}}$ is used to distinguish, for instance, the Frobenius norm or spectral norm. The space spanned by the columns of a matrix $M$ is an Euclidean subspace $\mathcal{M} \subset \Rn{n}$, and we denote it $\Span{M} \coloneqq \mathcal{M}$ for simplicity. The transpose of a real-valued matrix is denoted $M^T$. Orthogonality between two vectors $x$ and $y$ means that $\ip{x}{y} = 0$, and a unit (norm) vector $x$ satisfies $\norm{x} = 1$. An ($\ell_2$-)orthonormal or unitary matrix (or set of vectors) is such that $M^T M = I$. Non-singular matrices are a special case of square matrices for which the inverse $M^{-1}$ exists, and $M M^{-1} = I$. 

Central to this work are randomized embeddings for a given subspace of $\Rn{n}$. The following definition is taken from \cite{WoodruffComputationalAdvertisingTechniques2014}.
\begin{definition}[$\epsilon$-embedding]
    \label{def:epsembedd}
    Let $\K$ be a subspace of $\Rn{n}$ of dimension $\sa$. Then $\Om \in \Rnm{\dsk}{n}$ is called an $\epsilon$-embedding for $\K$ if for any $x \in \K$ it holds that
    \begin{align}
        \label{eq:epsembedd}
        (1-\epsilon) \normsq{x} \leq \normsq{\Om x} \leq (1+\epsilon) \normsq{x} \iff  \frac{1}{1+\epsilon} \normsq{\Om x} \leq \normsq{x} \leq  \frac{1}{1-\epsilon} \normsq{\Om x}
    \end{align}
    for a given $\epsilon \in [0,1]$.
\end{definition}
The dimension $\dsk$ is called the sketching size and is taken such that $\sa \leq \dsk \ll n$. An $\epsilon$-embedding preserves approximately the geometry of a subspace $\K \subset \Rn{n}$ while allowing computations in the much smaller subspace $\Rn{d}$. This definition is motivated by the Johnson-Lindenstrauss (JL) Lemma, which states that such an embedding exists for a given set of $\sa$ fixed vectors by taking $\dsk = O(\frac{\log \sa}{\epsilon^2})$; see \cite{Johnson1984ExtensionsLipschitzmappings}.
In contrast, an $\epsilon$-embedding satisfies \cref{eq:epsembedd} for a subspace $\K \subset \Rn{n}$, but to do so it requires $\dsk = O(\frac{\sa}{\epsilon^2})$. It preserves the singular values and the condition number of any basis for $\K$, as shown in \cite[Corollary 2.2]{BalabanovRandomizedGramSchmidt2022}.
\begin{corollary}
    \label{cor:embeddsvals}
    Let $\Om$ be an $\epsilon$-embedding for $\K$, and let $K$ be a matrix such that $\Span{K} = \K$. Then
    \begin{equation}
        \frac{1}{\sqrt{1+\epsilon}}\sigma_{\min}(\Om K) \leq \sigma_{\min}(K) \leq \sigma_{\max}(K) \leq \frac{1}{\sqrt{1-\epsilon}} \sigma_{\max}(\Om K)
    \end{equation}
    where $\sigma_{\min}$ and $\sigma_{\max}$ are the smallest and largest singular values, respectively.
\end{corollary}
The matrix $\Om K$ is called \emph{the sketch of $K$}. \Cref{cor:embeddsvals} is especially relevant when combined with the different randomized orthogonalization processes developed in recent years; see, for instance, Randomized Gram-Schmidt (RGS) from \cite{BalabanovRandomizedGramSchmidt2022}, randomized Householder QR (RHQR) from \cite{Grigori2024RandomizedHouseholderQR}, randomized Cholesky QR (RCholQR) in \cite{Balabanov2022RandomizedCholeskyQR}, or sketching and whitening from \cite{Rokhlinfastrandomizedalgorithm2008,Nakatsukasa2024FastAccurateRandomized}. These methods take as input a tall and skinny matrix $K \in \Rnm{n}{\sa}$ and output factors $Q \in \Rnm{n}{\sa}$ and $R \in \Rnm{\sa}{\sa}$ such that $K = QR$. While the QR decomposition produces an orthonormal factor $Q$ \cite{Golub2013Matrixcomputations}, randomized orthogonalization processes produce a \emph{sketch-orthonormal} factor $Q$, that is, the sketch of $Q$ is orthonormal, $(\Om Q)^T (\Om Q) = I$. Using \cref{cor:embeddsvals}, it holds that the condition number $\kappa(Q)$ satisfies 
\begin{equation}
    \label{eq:condsketched}
    \kappa(Q) \coloneqq \frac{\sigma_{\max}(Q)}{\sigma_{\min}(Q)} \leq \sqrt{\frac{1+\epsilon}{1-\epsilon}} \kappa(\Om Q).
\end{equation}
Hence, provided that $\Om Q$ is orthonormal, $\kappa(\Om Q) = 1$ results in $\kappa(Q) = O(1+\epsilon)$, which is relatively small and sufficient in many applications, while allowing faster computation in the smaller subspace $\Rn{\dsk}$ compared to the expensive computation of an orthonormal $Q$. Indeed, the authors in \cite{BalabanovRandomizedGramSchmidt2022} state that RGS does half of the floating point operations (flops) of Classical Gram-Schmidt (CGS), while having the stability of Modified Gram-Schmidt (MGS). RGS is also better suited for parallel implementation than CGS/MGS.

However, in many situations, the subspace $\K$ might be unknown at the beginning of the method, or we might encounter different subspaces, all of dimension $\sa$. In this case, \emph{oblivious subspace embeddings} are required.
\begin{definition}[Oblivious subspace embedding]
    The matrix $\Om \in \Rnm{d}{n}$ is said to be an oblivious subspace embedding of dimension $\sa$ if it is an $\epsilon$-embedding for every $\sa$ dimensional subspace with high probability (w.h.p.).
\end{definition}
Different oblivious subspace embeddings are presented in the literature, such as matrices with normalized Gaussian entries or sparse matrices with random signs; see \cite[Sections 8,9]{MartinssonRandomizednumericallinear2020}. They all have various advantages; for instance, Gaussian embeddings are well suited for parallel architectures, whereas sparse sign matrices are fast to apply to any vectors. They share a similar behavior, that is, being oblivious subspace embeddings w.h.p. when $\dsk = O(\frac{\sa}{\epsilon^2})$. This shows a trade-off between having a small $\epsilon$ to have good accuracy and having $\dsk$ the smallest possible to achieve major cuts in computing costs. In practice, it is advised to take $\dsk$ a small multiple of $\sa$, for example $\dsk = 2k$. This produces a distortion $\epsilon = \frac{1}{\sqrt{2}}$.

\section{Generalizing the randomized Arnoldi factorization}
\label{sec:rfactorizations}
This section introduces novel randomized factorizations for $\A$ and establishes their equivalence with the randomized Arnoldi factorization. All the factorizations considered are visually summarized in \cref{fig:tikzdecompos}. We also highlight the underlying Rayleigh-Ritz method that arises from the randomized Krylov framework.

\subsection{The randomized Krylov factorization}
Let $\A \in \Rnm{\ndim}{\ndim}$ be a square matrix. We are interested in solving the eigenvalue problem
\begin{equation}
    \label{eq:evp}
    \A \eivec = \eival \eivec,
\end{equation}
for the eigenpair $(\eival, \eivec)$, where $\eivec \in \Cn{\ndim}$ is a unit-norm eigenvector and $\eival \in \mathbb{C}$ is an eigenvalue. In many situations, only a few eigenpairs are sought, say $\sa$ of them, for example, those with largest or smallest modulus. 
Krylov subspaces can be used to approximate them. They are defined as 
\begin{equation}
    \label{eq:krylovsubspace}
    \K_\sa(\A, v_1) \coloneqq \Span{v_1, \A v_1, \dots, \A^{\sa -1} v_1},
\end{equation}
given a starting unit vector $v_1$ and a target dimension $\sa$. We denote $\K_\sa(\A, v_1) = \K_\sa$ for brevity. In this article, we assume that the dimension of any Krylov subspace $\K_\sa$ is exactly $\sa$ and not less, so that we can work with bases of full rank. If this is not the case, it means that an invariant subspace for $\A$ has been found and exact eigenpairs can be extracted, so the eigenproblem is solved. In practice, the starting unit vector $v_1$ is taken at random and it is unlikely that an invariant subspace will be found. Constructing an orthonormal or sketch-orthonormal basis for $\K$, as considered in \cite{saad2011numerical} and \cite{Nakatsukasa2024FastAccurateRandomized,Balabanov2021RandomizedblockGram, Damas2024RandomizedImplicitlyRestarted}, respectively, is useful in the setting of a (randomized) Rayleigh-Ritz procedure that extracts $\sa$ approximate eigenpairs from $\A$. 

We start by defining a randomized Arnoldi factorization of size $\sa$ for $\A$, which naturally arises when constructing a sketch-orthonormal basis for the ill-conditioned basis $\{v_1, \A v_1, \dots, \A^{\sa -1} v_1\}$.
\begin{definition}[Randomized Arnoldi factorization of size $\sa$]
    \label{def:rArno}
    Assume that $\Om$ is an $\epsilon$-embedding for $\K_{\sa+1}$. A randomized Arnoldi factorization of size $\sa$ for $\A$ is the decomposition
    \begin{equation}
    \label{eq:rArno}
        \A V = V H + \beta v_{\sa+1} e_{\sa}^T
    \end{equation}
    where 
    \begin{itemize}
        \item $H \in \Rnm{\sa}{\sa}$ is upper-Hessenberg, i.e.\@, upper-triangular with possible sub-diagonal entries. It is said to be unreduced if all sub-diagonal entries are nonzero.
        \item $(V \; v_{\sa+1})$ is a sketch-orthonormal set of $\sa+1$ vectors in $\Rn{\ndim}$, that is, $(\Om V \; \Om v_{\sa+1})^T (\Om V \; \Om v_{\sa+1}) = I_{\sa+1}$. Moreover, $(V \; v_{\sa+1})$ is a basis for $\K_{\sa+1}$.
    \end{itemize} 
\end{definition}
\begin{remark}
    Throughout this section, $\Om$ is assumed to be an $\epsilon$-embedding for $\K_{\sa+1}$. Our aim in what follows is to show that the span of $(V \; v_{\sa+1})$, which is $\K_{\sa+1}$, is preserved through some transformations applied to obtain different factorizations. Hence, having an $\epsilon$-embedding for $\K_{\sa+1}$ is sufficient for all factorizations. 
\end{remark}
Note that the deterministic Arnoldi factorization is written the same way, with $(V \; v_{\sa+1})$ being $\ell_2$-orthonormal. In the early 2000s, G. W. Stewart introduced the Krylov decomposition to simplify restart procedures and deflation techniques for methods such as the deterministic Implicitly Restarted Arnoldi from \cite{Sorensen1992ImplicitApplicationPolynomial,Lehoucq1995DeflationTechniquesImplicitly}. We recall the definition presented in \cite{Stewart2002KrylovSchurAlgorithm}:
\begin{definition}[Krylov decomposition of order $\sa$]
    \label{def:KD}
    A Krylov decomposition of order $\sa$ for $\A$ is the decomposition
    \begin{equation}
    \label{eq:KD}
        \A \U = \U \Bsa + \ulast \blast\tsp
    \end{equation}
    where $\Bsa \in \Rnm{\sa}{\sa}$ must be of rank $\sa$, and $(\U \; \ulast)$ is an independent set of $\sa+1$ vectors in $\Rn{\ndim}$. The span of $(\U \; \ulast)$ is called the space of the decomposition.
\end{definition}
It is a generalization of the (randomized) Arnoldi decomposition in the sense that a (randomized) Arnoldi decomposition is a Krylov decomposition. \Cref{def:KD} removes any structural hypotheses on $\Bsa$ and $\blast$, allowing for a more flexible framework. 
In this work, we consider Krylov decompositions where the Krylov basis $\U$ is sketch-orthonormal.
\begin{definition}[Sketch-orthonormal Krylov decomposition of order $\sa$]
    \label{def:sketchorthKD}
    A sketch-orthonormal Krylov decomposition of order $\sa$ for $\A$ is the Krylov decomposition
    \begin{equation}
    \label{eq:sketchorthKD}
        \A \U = \U \Bsa + \ulast \blast\tsp,
    \end{equation}
    where $(\U \; \ulast)$ is sketch-orthonormal.
\end{definition}
In the following, we justify the name sketch-orthonormal Krylov decomposition by showing that it is equivalent to a randomized Arnoldi factorization, and discuss its usefulness for obtaining spectral information on $\A$ within a sketch Rayleigh-Ritz procedure.

\subsection{Sketch-orthonormal Krylov subspaces within the Rayleigh-Ritz method}
\label{sec:KSasRR}
The \rArno{} factorization in \cref{eq:rArno} forms the foundation of a randomized Rayleigh-Ritz procedure (see, e.g., \cite{Nakatsukasa2024FastAccurateRandomized,Balabanov2021RandomizedblockGram}), with $\Hsa$ representing $\A$ with respect to the sketch-orthonormal Krylov basis. Here, we address how $\Bsa$ from the more general sketch-orthonormal Krylov decomposition \cref{eq:sketchorthKD} can be interpreted as a Rayleigh quotient. Using definitions from \cite{Stewart2002AddendumKrylovSchur}, recall that the Rayleigh-Ritz method for an operator $\A$ is based on two test subspaces $\K$ and $\Lspace$. If $V$ and $W$ are matrix representations of bases for $\K$ and $\Lspace$, respectively, then the Rayleigh quotient $B$ with respect to $V$ and $W$ is defined as
\begin{equation}
    B \coloneqq (W^TV)^\dag W^T A V.
\end{equation} 
The only requirements are that $\dim(\Lspace) \geq \dim(\K)$ and that $W^TV$ is non-singular, so $(W^TV)^\dag$ is a left-inverse for $(W^TV)$. If $(\eival,V \yrivec)$ is an eigenpair for $\A$, then
\begin{equation}
    B \yrivec = (W^TV)^\dag W^T \A V \yrivec = \eival (W^TV)^\dag W^T  V \yrivec = \eival \yrivec,
\end{equation}
so $(\eival, \yrivec)$ is an eigenpair for $B$. Conversely, if $(\eival, \yrivec)$ is an eigenpair for $B$, then $B \yrivec = (W^TV)^\dag W^T \A V \yrivec = \eival \yrivec$, so
\begin{equation}
    \label{eq:ripairs}
    V (W^TV)^\dag W^T \A V \yrivec = P_{\K}^{\Lspace} \A \eivec = \eival \eivec,
\end{equation}
where $\eivec \coloneqq V \yrivec$. Thus, $(\eival, V \yrivec)$ is an eigenpair for the approximate operator $P_{\K}^{\Lspace} A$, where $P_{\K}^{\Lspace}$ is an oblique projector onto $\K$. Equivalently, $P_{\K}^{\Lspace}(\A \eivec - \eival \eivec) = 0$, which is often called the Petrov-Galerkin condition.

In the setting of a sketch-orthonormal Krylov decomposition, we have
\begin{align}
    \label{eq:Bsa_is_RR}
    \A \U = \U \Bsa + \ulast \blast\tsp \implies (\Om \U)^\dag \Om \A \U = \Bsa 
\end{align}
by left-multiplying with $(\Om \U)^\dag \Om$ and using the sketch-orthonormality of $(\U \; \ulast)$. Thus, $\Bsa$ can be identified as a Rayleigh quotient where $\K$ is the Krylov subspace $\Span{\U}$, and $\Lspace$ is $\Span{\Om^T}$, with $V = \U$ and $W = \Om^T$. The following theorem summarizes the situation.
\begin{theorem}
    Assume a sketch-orthonormal Krylov decomposition, i.e.\@, $\A \U = \U \Bsa + \ulast \blast\tsp$ with $(\U \; \ulast )$ sketch-orthonormal.
    Then $\Bsa$ is a Rayleigh quotient with respect to $\U$ and $\Om^T$, with
    \begin{equation}
        \label{eq:rayleighquo-rKD}
        \Bsa = (\Om \U)^\dag \Om \A \U.
    \end{equation}
    In this situation, the underlying oblique projector $P_{\K}^{\Lspace}$ on $\K$ is 
    \begin{equation}
        P_{\K}^{\Lspace} = \U (\Om \U)^\dag \Om
    \end{equation}
    and any eigenpair $(\eival, \yrivec)$ of $\Bsa$ satisfies a sketched Petrov-Galerkin condition, 
    \begin{equation}
        \label{eq:rPetrovGal}
        \ip{\Om (\A \U \yrivec - \eival \U \yrivec)}{\Om v} = 0, \; \forall v \in \K.
    \end{equation}
\end{theorem}
\begin{proof}
    The interpretation of $\Bsa$ as a Rayleigh quotient is shown above in \cref{eq:Bsa_is_RR}. Note that $\dim(\Lspace) = \dim(\Span{\Om^T}) = \rank(\Om) \geq \rank(\Om \U)$, and by the $\epsilon$-embedding property of $\Om$ for $\U$, we have $\rank(\Om \U) = \sa$. Thus, $\dim(\Lspace) \geq \dim(\K)$. In addition, $W^TV = \Om \U$ is non-singular since $\U$ is full-rank and $\Om$ is an $\epsilon$-embedding for $\Span{\U}$. The expression for $P_{\K}^{\Lspace}$ follows directly for $V = \U$ and $W = \Om^T$:
    $$P_{\K}^{\Lspace} =  V (W^TV)^\dag W^T = \U (\Om \U)^\dag \Om. $$
    Moreover, $(\Om \U)^\dag = (\Om \U)^T$ since $\Om \U$ is orthonormal. If $(\eival, \yrivec)$ is an eigenpair for $\Bsa$, \cref{eq:ripairs} gives
    \begin{align*}
        P_{\K}^{\Lspace} \A \U  \yrivec & = \eival \U  \yrivec \\
        \U (\Om \U)^T \Om \A \U  \yrivec - \eival \U  \yrivec & = 0 & \text{expressing $P_{\K}^{\Lspace}$} \\
        \U [( \Om \U)^T \Om \A \U  \yrivec - \eival  \yrivec] & = 0 &  \text{factorizing $\U$} \\
        (\Om \U)^T \Om \A \U  \yrivec - \eival  \yrivec & = 0 &  \text{since columns of $\U$ are independent}  \\
        (\Om \U)^T (\Om \A \U  \yrivec - \eival  \Om \U \yrivec) & = 0  &  \text{factorizing $(\Om \U)^T$ and writing $\yrivec = I_\sa \yrivec = (\Om \U)^T \Om \U \yrivec$ } \\
        (\Om \U)^T \Om (\A \U  \yrivec - \eival \U \yrivec) & = 0  &  \text{factorizing $\Om$} \\
    \end{align*}
    which shows that the sketch of $\A \U  \yrivec - \eival \U \yrivec$ is orthogonal to the sketch of $\U$, i.e.\@, $\Om (\A \U \yrivec - \eival \U \yrivec) \perp \Om \K$.
\end{proof}

The Rayleigh-Ritz method is central to solving the eigenvalue problem for a few eigenpairs in numerical linear algebra. It reduces the size of the problem by solving a smaller eigenvalue problem for $\Bsa$. The hope is that the approximate operator $P_\K^\Lspace \A$, for which exact eigenpairs are computed, is close to $\A$ when restricted to the $\sa$-dimensional eigenspace of interest. This heavily depends on the subspace $\K$, and Krylov subspaces and the Arnoldi method have been highly successful, as reviewed in \cite{saad2011numerical}. In the deterministic framework, one works with an orthogonal basis $V_\sa$ for $\K$, resulting in an orthogonal projector $P_\K^\K = V_\sa V_\sa^T$. Using a randomized approach modifies the projection by making it oblique and transferring the $\ell_2$-orthogonality into the sketched subspaces, as seen in \cref{eq:rPetrovGal}, but it remains a projection into the well-studied Krylov subspaces, which are rich in spectral information about $\A$. Thus, good performance can be expected from randomized eigensolvers using Krylov subspaces, as developed in this article.

\subsection{An equivalence between Krylov decompositions and randomized Arnoldi decompositions}

We state a central result that links any Krylov decomposition to a randomized Arnoldi decomposition, further reinforcing the idea that the randomized Krylov decompositions framework is relevant for the eigenvalue problem.
\begin{theorem}
    \label{th:KDeqRA}
    A Krylov decomposition as defined in \Cref{def:KD} is equivalent to a randomized Arnoldi decomposition in the sense that both decompositions span the same space, for which $\Om$ is an $\epsilon$-embedding. That is, if \cref{eq:KD} holds:
    $$\A \U = \U \Bsa + \ulast \blast\tsp,$$
    then there exists a sketch-orthonormal $V$, an upper-Hessenberg $H$, a vector $v_{\sa+1}$ and a scalar $\beta$ such that 
    $$\A V = V H + \beta v_{\sa+1} e_{\sa+1}^T$$ 
    with $\Span{\U \; \ulast} = \Span{V \; v_{\sa+1}}$. If $H$ is unreduced, the randomized Arnoldi factorization is essentially unique, that is, up to the signs of the sub-diagonal elements of $H$ and vectors $v_i$.
\end{theorem}
As pointed out in the original work of \cite{Stewart2002KrylovSchurAlgorithm} which is set in the deterministic setting, this result legitimates the name "Krylov decomposition" since $\Span{\U \; \ulast} = \Span{V \; v_{\sa+1}} = \K_{\sa+1}$. The proof, delayed in the following, is based on similarity transformations and translations that can be applied to \cref{eq:KD} while maintaining the original span of $(\U \; \ulast)$. These operations have been introduced by the author of \cite{Stewart2002KrylovSchurAlgorithm} and are:
\begin{itemize}
    \item For similarity: if $R \in \Rnm{\sa}{\sa}$ is non-singular, then right multiply \cref{eq:KD} by $R^{-1}$ to obtain 
    \begin{align*}
        \A \U R^{-1} & = \U \Bsa R^{-1} + \ulast \blast\tsp R^{-1} \\
                     & = \U R^{-1} R \Bsa R^{-1} + \ulast \blast\tsp R^{-1} \\
        \implies  \A \tU &  = \tU \tBsa + \ulast \blast\tsp
    \end{align*}
    with $ \tU \coloneqq \U R^{-1}$, $\tBsa \coloneqq R \Bsa R^{-1}$ and $\tblast\tsp \coloneqq \blast\tsp R^{-1}$. The last equation is a Krylov decomposition, and we have $\Span{\tU} = \Span{\U}$, thus $\Span{\tU \; \ulast} = \Span{\U \; \ulast} $. Moreover $\tBsa$ and $\Bsa$ are similar so they contain the same spectral information.
    \item For translation: if $\alpha \tulast \coloneqq \ulast - \U g$ for any $g \in \Rn{\sa}$ and any scalar $\alpha$, then by substitution in \cref{eq:KD}
    \begin{align*}
         \A \U & = \U \Bsa + (\alpha \tulast + \U g) \blast\tsp \\
            & = \U (\Bsa + g \blast\tsp) + \alpha \tulast \blast\tsp \\
            & = \U \tBsa + \tulast \tblast\tsp
    \end{align*}
    with $\tBsa \coloneqq \Bsa + g \blast\tsp$ and $\tblast \coloneqq \alpha \blast$. The last equation is also a Krylov decomposition and since $\tulast \in \Span{\U \; \ulast}$, it holds that $\Span{\U \; \tulast} = \Span{\U \; \ulast}$.
\end{itemize}
To prove \Cref{th:KDeqRA}, we first show two lemmas based on these transformations.
\begin{lemma}
    \label{lem:rorthoU}
    Assume that \cref{eq:KD} holds, that is $\A \U = \U \Bsa + \ulast \blast\tsp$. Then an equivalent Krylov decomposition, where the factor $\tU$ is sketch-orthonormal, can be obtained by means of a similarity transformation.
\end{lemma}
\begin{proof}
    We start with $\A \U = \U \Bsa + \ulast \blast\tsp$. Then we can perform any randomized QR factorization on $\U$ to obtain 
    $$\U = \tU R,$$
    with $\tU \in \Rnm{\ndim}{\sa}$ being sketch-orthonormal, that is,$(\Om \tU)^T (\Om \tU)  = I_\sa$. Thus $\U R^{-1} = \tU$, so using $R^{-1}$ as a similarity transform, which is non-singular if $\U$ is full rank, gives the stated result.
\end{proof}
\begin{lemma}
    \label{lem:rorthou}
    Assume that \cref{eq:KD} holds, that is $\A \U = \U \Bsa + \ulast \blast\tsp$. Assume also that $\U$ is sketch-orthonormal. Then an equivalent Krylov decomposition, where the factor $\tulast$ is sketch-orthogonal to $\U$, can be obtained by means of a translation. Moreover, the sketch-norm of $\tulast$ can be set to $1$.
\end{lemma}
\begin{proof}
    We know that we can apply a translation of the form $\alpha \tulast \coloneqq \ulast - \U g$ to \cref{eq:KD}. In order to make $\tulast$ sketch-orthogonal to $\U$, we set
    $$g \coloneqq (\Om \U)^\dag (\Om \ulast) = \arg \min_{y} \norm{\Om(\U y - \ulast)}.$$
    This gives $\alpha \tulast  = (I - \U (\Om \U)^\dag \Om) \ulast = (I-\POm) \ulast$ with $\POm \coloneqq \U (\Om \U)^\dag \Om$ the matrix representation of an oblique projector on $\Span{\U} = \K_\sa$. In other words, we retrieve from $\ulast$ its oblique projection on $\K_\sa$.
    By doing so, it holds that
    \begin{align*}
        (\Om \U)^T (\Om \tulast) & = \frac{1}{\alpha} (\Om \U)^T [\Om (\ulast - \U (\Om \U)^T \Om \ulast)] \\
        & =  \frac{1}{\alpha} (\Om \U)^T [\Om \ulast - \Om \U (\Om \U)^T \Om \ulast] \\
        & = \frac{1}{\alpha}  [(\Om \U)^T \Om \ulast - (\Om \U)^T  (\Om \U) (\Om \U)^T \Om \ulast]\\
        & =  \frac{1}{\alpha}  [(\Om \U)^T \Om \ulast - (\Om \U)^T \Om \ulast] = 0.
    \end{align*}
    Lastly, the scalar $\alpha$ can be set to make the sketch-norm of $\tulast$  equal to $1$ with
    \begin{equation*}
        \alpha \coloneqq \norm{\Om(\ulast - \U g)} \implies \norm{\Om \tulast} = 1.
    \end{equation*}
\end{proof}

Combining \Cref{lem:rorthoU} and \Cref{lem:rorthou} allows us to obtain a sketch-orthonormal Krylov decomposition, which we introduced above in \cref{def:sketchorthKD}. 

We are now ready to derive the proof of \cref{th:KDeqRA} by using the above transformations and lemmas.
\begin{proof}[Proof of \cref{th:KDeqRA}]
    \label{proof:KDeqRA}
    We start from the Krylov decomposition $\A \U = \U \Bsa + \ulast \blast\tsp$. Following the idea of the proof in \cite[Theorem 2.2]{Stewart2002KrylovSchurAlgorithm}, we apply the following transformations in sequence to \cref{eq:KD}:
    \begin{enumerate}
        \item Use a similarity transformation to make $\U$ sketch-orthonormal using \cref{lem:rorthoU}.
        \item Use a translation to make $\ulast$ sketch-orthogonal to $\U$ and with a sketch-norm of $1$ using \cref{lem:rorthou}.
        \item Reduce $\blast$ to $\beta e_\sa$ using a non-singular unitary similarity transformation, namely a Householder reflector as derived in \cite[Chapter 5]{Golub2013Matrixcomputations}. 
        \item Lastly, reduce $\Bsa$ to Hessenberg form. As stated in \cite{Stewart2002KrylovSchurAlgorithm}, this can be done by unitary similarity transformations, namely Householder reflectors again, without adding non-zeros entries to $\beta e_\sa$.
    \end{enumerate}
    Eventually, a randomized Arnoldi factorization of the form $\A V = V H + \beta v_{\sa+1} e_{\sa+1}^T$ is obtained, and if $H$ is unreduced, then the factorization is essentially unique for the given $\Om$ up to signs of $h_{i+1,i}$ and $v_i$  due to a property of the randomized Arnoldi decomposition, see \cite[Randomized Implicit Q theorem]{Damas2024RandomizedImplicitlyRestarted}. 
\end{proof}

\subsection{The \rKS{} decomposition}

To conclude this section, we introduce the \rKS{} decomposition that is essential for the \rKS{} algorithm. The main idea of (r)KS is to enhance the restart procedure of (r)IRA, which is based on applying several shifted QR steps to the Hessenberg matrix. This strategy can be unstable because of the shifts and it does not entail a simple deflation procedure. However, it can be performed equivalently by moving around the eigenvalues of the Schur form of $\Bsa$ in a  stable way using only unitary transformations. We recall the real Schur form of a matrix  from \cite{Golub2013Matrixcomputations}:
\begin{proposition}[Real Schur form]
    \label{prop:realSchur}
    If $M \in \Rnm{\ndim}{\ndim}$, then there exists an orthonormal $Q  \in \Rnm{\ndim}{\ndim}$ such that 
    \begin{equation}
        \label{eq:realSchur}
        M = QTQ^T,
    \end{equation}
    where $T \in \Rnm{\ndim}{\ndim}$ is block upper-triangular with block of size up to $2 \times 2$ accounting for complex conjugate eigenvalues of $M$.  The columns of $ Q $, known as Schur vectors, can be arranged so that the eigenvalues $ \eival_i $ of $ M $ appear in any order along the diagonal of $ T $.
\end{proposition}
Note that we can reduce $\Bsa$ to Schur form through a similarity relation: if $\Bsa = Q \Tsa Q^T$ then right-multiplying \cref{eq:KD} by $Q$ gives
\begin{equation}
    \A (\U Q) = (\U Q) \Tsa + \ulast \blast\tsp Q
\end{equation}
which is an equivalent Krylov decomposition. Combined with the fact that we can always obtain a sketch-orthonormal Krylov decomposition as discussed in the proof of \cref{th:KDeqRA}, we obtain the following definition.
\begin{definition}[\rKS{} decomposition]
    \label{def:rKS}
    A \rKS{} decomposition of order $\sa$ for $\A$ is the decomposition
    \begin{equation}
    \label{eq:rKS}
        \A \U = \U \Tsa + \ulast \blast\tsp,
    \end{equation}
    where $\Tsa \in \Rnm{\sa}{\sa}$ is of rank $\sa$ and block upper-triangular as defined in \cref{prop:realSchur}, and $(\U \; \ulast)$ is a $\sa+1$ dimensional sketch-orthonormal basis for $\K_{\sa+1} \subset \Rn{\ndim}$.
\end{definition}
If \cref{eq:KD} holds, then a decomposition as \cref{eq:rKS} always exists. The different decompositions that we use are visually summarized in \cref{fig:tikzdecompos}. The symbol $\Omperpornot$ means that the associated bases, noted $U$ -or $V$ in the Arnoldi setting-, can be either $\ell_2$-orthonormal or sketch-orthonormal depending on the framework considered.

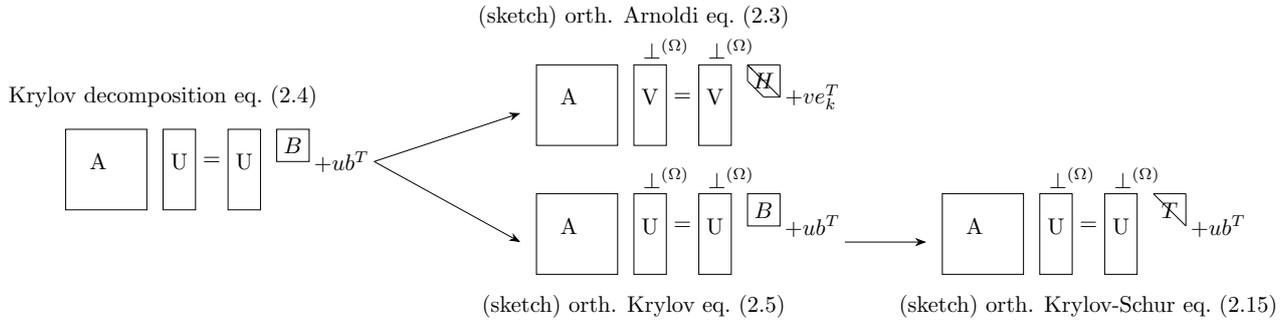
\begin{figure}[!ht]
\centering
\resizebox{1\textwidth}{!}{%
\begin{circuitikz}
\tikzstyle{every node}=[font=\normalsize]
\node [font=\normalsize] at (6.5,12.25) {Krylov decomposition \cref{eq:KD}};
\draw  (5,11.75) rectangle (6.25,10.5);
\draw  (6.5,11.75) rectangle (7,10.5);
\node [font=\normalsize] at (7.25,11.25) {=};
\draw  (7.5,11.75) rectangle (8,10.5);
\node [font=\normalsize] at (5.5,11.25) {A};
\node [font=\normalsize] at (6.75,11.25) {U};
\node [font=\normalsize] at (9.25,11.25) {$+ ub^T$};
\node [font=\normalsize] at (7.75,11.25) {U};
\draw  (8.25,11.75) rectangle (8.75,11.25);
\node [font=\normalsize] at (8.5,11.5) {$B$};
\draw [->, >=Stealth] (9.75,11.25) -- (12,12);

\draw [->, >=Stealth] (9.75,11.25) -- (12,10);
\node [font=\normalsize] at (13.75,13.5) {(sketch) orth. Arnoldi \cref{eq:rArno}};
\draw  (12.25,12.75) rectangle (13.5,11.5);
\draw  (13.75,12.75) rectangle (14.25,11.5);
\node [font=\normalsize] at (14.5,12.25) {=};
\draw  (14.75,12.75) rectangle (15.25,11.5);
\node [font=\normalsize] at (12.75,12.25) {A};
\node [font=\normalsize] at (14,12.25) {V};
\node [font=\normalsize] at (16.5,12.25) {$+ v e_k^T$};
\node [font=\normalsize] at (15,12.25) {V};
\node [font=\normalsize] at (15.75,12.5) {$H
$};
\node [font=\normalsize] at (14.25,13) {$\Omperpornot$};
\node [font=\normalsize] at (15.25,13) {$\Omperpornot$};
\draw [short] (15.5,12.75) -- (16,12.75);
\draw [short] (16,12.75) -- (16,12.25);
\draw [short] (15.5,12.75) -- (16,12.25);
\draw [dashed] (15.5,12.75) -- (15.75,12.5);
\draw [short] (15.5,12.75) -- (15.5,12.5);
\draw [short] (15.5,12.5) -- (15.75,12.25);
\draw [short] (15.75,12.25) -- (16,12.25);
\node [font=\normalsize] at (13.75,9) {(sketch) orth. Krylov \cref{eq:sketchorthKD}};
\draw  (12.25,10.75) rectangle (13.5,9.5);
\draw  (13.75,10.75) rectangle (14.25,9.5);
\node [font=\normalsize] at (14.5,10.25) {=};
\draw  (14.75,10.75) rectangle (15.25,9.5);
\node [font=\normalsize] at (12.75,10.25) {A};
\node [font=\normalsize] at (14,10.25) {U};
\node [font=\normalsize] at (16.5,10.25) {$+ u b^T$};
\node [font=\normalsize] at (15,10.25) {U};
\draw  (15.5,10.75) rectangle (16,10.25);
\node [font=\normalsize] at (15.75,10.5) {$B$};
\node [font=\normalsize] at (14.25,11) {$\Omperpornot$};
\node [font=\normalsize] at (15.25,11) {$\Omperpornot$};
\draw [->, >=Stealth] (17,10) -- (18.25,10);
\node [font=\normalsize] at (20.75,9) {(sketch) orth. Krylov-Schur \cref{eq:rKS}};
\draw  (18.5,10.75) rectangle (19.75,9.5);
\draw  (20,10.75) rectangle (20.5,9.5);
\node [font=\normalsize] at (20.75,10.25) {=};
\draw  (21,10.75) rectangle (21.5,9.5);
\node [font=\normalsize] at (19,10.25) {A};
\node [font=\normalsize] at (20.25,10.25) {U};
\node [font=\normalsize] at (22.75,10.25) {$+ u b^T$};
\node [font=\normalsize] at (21.25,10.25) {U};
\node [font=\normalsize] at (22,10.5) {$T$};
\node [font=\normalsize] at (20.5,11) {$\Omperpornot$};
\node [font=\normalsize] at (21.5,11) {$\Omperpornot$};
\draw  (21.75,10.75) rectangle (22.25,10.75);
\draw  (22.25,10.75) rectangle (22.25,10.25);
\draw [short] (21.75,10.75) -- (22.25,10.25);
\end{circuitikz}
}%
\caption{Different decompositions of $\A$ based on Krylov subspaces}
\label{fig:tikzdecompos}
\end{figure}

\section{The \rKS{} method}
\label{sec:rKS}
In this section, we introduce the randomized Krylov-Schur eigensolver, presented in \Cref{alg:rKS}. We show that an iteration of rKS is equivalent to an iteration of rIRA, and describe a structure-preserving sketched re-orthogonalization process.

\subsection{Deriving the algorithm}
The main idea behind the Krylov-Schur (KS) algorithm is to simplify the implementation of the Implicitly Restarted Arnoldi (IRA) method with exact shifts. IRA with exact shifts can be summarized as follows: given a target number $\sa$ of eigenvalues and a total Krylov dimension $\ba$, at each iteration an Arnoldi factorization of size $\ba$ is computed. Then, $\nshi$ unwanted eigenvalues are selected among the $\ba$ eigenvalues of the Hessenberg matrix and are removed from $\Hba$ using the implicit shifted QR algorithm described in \cite{Golub2013Matrixcomputations,Sorensen1992ImplicitApplicationPolynomial,Lehoucq1995DeflationTechniquesImplicitly}. This is called the contraction step, resulting in an Arnoldi factorization of size $\sa$. Finally, the factorization is expanded back to size $\ba$ in the expansion step. In theory, the errors on the $\sa$ desired eigenpairs decrease after each contraction, which is equivalent to polynomial filtering (see \cref{sec:equivKSIRA}).

However, the contraction step is not always stable in practice, as shown in \cite[Chapter 5]{Lehoucq1995DeflationTechniquesImplicitly}, where a loss of forward stability of the implicit shifted QR method is demonstrated. In contrast, significant progress has been made in stable reordering of Schur decompositions; see \cite{Bai1993swappingdiagonalblocks} for LAPACK algorithms, and \cite{Kressner2006Blockalgorithmsreordering} for block versions. The deterministic KS algorithm uses Schur forms to perform the contraction step more reliably than shifted QR steps. Moreover, deflation techniques can be implemented using Schur form reordering rather than more complex methods in IRA, making KS efficient and practical. Below, we show how these ideas from \cite{Stewart2002KrylovSchurAlgorithm} extend to the randomized setting, deriving both contraction and expansion steps.

Consider a \rKS{} factorization of length $\ba$ as in \Cref{def:rKS},
\begin{equation}
\label{eq:rKS-ba}
    \A \Uba = \Uba \Tba + \uba \bba\tsp,
\end{equation}
where $\Om$ is an $\epsilon$-embedding for $\K_{\ba+1}$. Any truncation of \cref{eq:rKS-ba} is also a \rKS{} factorization. For example, partition as follows:
\begin{equation}
    \label{eq:partioning_Uk_Up}
    \A ( \U \; U_\nshi ) = ( \U  \; U_\nshi ) 
    \begin{pmatrix}
        \Tsa & \tilde{T} \\
        0 & T_\nshi
    \end{pmatrix}
    + \uba (\blast\tsp \; b_\nshi\tsp).
\end{equation}
Truncating to the first $\sa$ columns yields:
\begin{equation}
    \label{eq:truncate-sa-rksfacto}
    \A \U = \U \Tsa +  \uba \blast\tsp.
\end{equation}
It is clear that $\U$ is sketch-orthonormal since $\Uba = ( \U \; U_\nshi )$ is, and $\uba$ is sketch-orthogonal to $\U$. Moreover, $\Tsa$ is block upper triangular as in \cref{prop:realSchur}, provided the partitioning in \cref{eq:partioning_Uk_Up} keeps the eigenvalues of $\Tsa$ together under complex conjugation. The rKS algorithm proceeds as follows:
\begin{enumerate}
    \item (Contraction step) From a randomized Krylov-Schur factorization of size $\ba$ as in \cref{eq:rKS-ba}, use stable methods to reorder the Schur form and move $\nshi$ unwanted eigenvalues of $\Tba$ to its lower-right block $T_\nshi$. This is equivalent to placing the $\sa$ desired eigenvalues in the upper-left block $\Tsa$.
    \item (Contraction step) Truncate the factorization to size $\sa$ as in \cref{eq:truncate-sa-rksfacto}.
    \item (Expansion step) Extend this factorization to size $\ba$ using randomized Arnoldi, see \cref{alg:extend-rKS}.
    \item Obtain a \rKS{} factorization of size $\ba$ from the previous randomized Krylov factorization via a unitary transformation.
\end{enumerate}
The method is described in \cref{alg:rKS}, where indices are omitted except for the Rayleigh quotient matrices $B$ or $T$ for clarity. A visual representation is given in \cref{fig:tikzrKS}.
\begin{algorithm}[!htbp]
    \caption{randomized Krylov-Schur (rKS)}
    \label{alg:rKS}
    \begin{algorithmic}[1]
        \Require $\A \in \Rnm{n}{n}$, $\sa$ the number of wanted eigenvalues, $\ba$ the Krylov dimension, $\Om \in \Rnm{\dsk}{n}$ an  $\epsilon$-embedding for $\K_{\ba+1} $, $\uinit \in \Rn{n}$ with $\norm{\Omega \uinit} =1$ and a residual error tolerance $\eta$. 
        \Ensure A \rKS{} decomposition $\A \tilde{U} = \tilde{U} \tilde{\Tsa} + \tilde{u} \tilde{b}\tsp$ such that $\enorm{\A (\tilde{U} Y) - (\tilde{U} Y) \tilde{\Lambda}_k} \leq \eta $ where $\tilde{U} Y$ and $\tilde{\Lambda}_k$ are the sought $\sa$ Ritz eigenvectors and Ritz eigenvalues, respectively.
        \vspace{0.2cm}
        \State Perform $\ba$ steps of the \rArno{} procedure starting from $\uinit$ to obtain a sketch-orthonormal Krylov decomposition
        $\A U = U \Bba + u b\tsp$ and  $(S \; s) = (\Om U \; \Om u)$ \label{alg:line:rKS-arnoldi-Init}
        \While{convergence not declared} 
        \State Reduce $\Bba$ to Schur form with $\Bba = Q_1 \Tba Q_1^T$ to obtain a \rKS{} decomposition $$\A U Q_1 = U Q_1 \Tba + u b\tsp Q_1$$
        \vspace{-0.2cm}
        \State Move around unwanted eigenvalues to the bottom-right part of $\Tba$ using an orthonormal similarity transformation $Q_2$ s.t.\@ $\tTba = Q_2^T \Tba Q_2$ to obtain an ordered \rKS{} decomposition 
        \begin{equation*}
            \A U Q_1 Q_2 = U Q_1 Q_2 \tTba + u b\tsp Q_1 Q_2.
        \end{equation*}
        \vspace{-0.2cm}
        \State Truncate this \rKS{} factorization with $\tilde{U} \coloneqq U Q_1 Q_2 (e_1, \dots, e_\sa)$, $\tilde{u} \coloneqq u$, $\tilde{b} \coloneqq b\tsp Q_1 Q_2(e_1, \dots, e_\sa)$ and $\tilde{S} = S Q_1 Q_2(e_1, \dots, e_\sa)$, $\tilde{s} = s$ to obtain the contracted \rKS{} factorization $$\A \tilde{U} = \tilde{U} \tilde{\Tsa} + \tilde{u} \tilde{b}\tsp.$$ \label{alg:line:rKS-truncate}
        \vspace{-0.2cm}
        \State Compute the eigendecomposition of the small factor $\tTsa Y = Y \tilde{\Lambda}_k$ and obtain the residual errors using \cref{sec:errormonitoring}. \label{alg:line:rksreserrors}
        \State If the tolerance $\eta$ is not satisfied, extend this factorization to a length $\ba$ sketch-orthonormal Krylov decomposition $\A U = U \Bba + u b\tsp$ and  $(S \; s) = (\Om U \; \Om u)$ by using the \rArno{} procedure from \cref{alg:extend-rKS}. \label{alg:line:extend-facto}
        \EndWhile 
    \end{algorithmic}
\end{algorithm}

\begin{figure}[!ht]
\centering
\resizebox{0.8\textwidth}{!}{%
\begin{circuitikz}
\tikzstyle{every node}=[font=\LARGE]
\draw [<->, >=Stealth] (23.5,15.25) .. controls (24.25,16.75) and (26,15.25) .. (24.25,14.5);
\node [font=\LARGE] at (8.25,14.25) {$A$};
\draw  (8.75,15.5) rectangle (10,13);
\node [font=\LARGE] at (10.75,14.25) {=};
\draw  (11.25,15.5) rectangle (12.5,13);
\draw  (12.5,15.5) rectangle (12.75,13);
\node [font=\LARGE] at (13,14.25) {+};
\draw [short] (13.75,15.5) -- (15,15.5);
\draw [short] (13.75,15.5) -- (15,14.25);
\draw [short] (15,15.5) -- (15,14.25);
\draw [short] (13.75,14.25) -- (15,14.25);
\draw [short] (15,14.25) -- (15,14);
\draw [short] (15,14) -- (13.75,14);
\draw [short] (13.75,14) -- (13.75,14.25);
\draw [->, >=Stealth] (15,13.25) -- (16.75,13.25);
\node [font=\LARGE] at (17.5,14.25) {$A$};
\draw  (18.25,15.5) rectangle (19.5,13);
\node [font=\LARGE] at (20,14.25) {=};
\draw  (20.75,15.5) rectangle (22,13);
\draw  (22,15.5) rectangle (22.25,13);
\node [font=\LARGE] at (22.5,14.25) {+};
\draw [short] (23.25,15.5) -- (24.5,15.5);
\draw [short] (23.25,15.5) -- (24.5,14.25);
\draw [short] (24.5,15.5) -- (24.5,14.25);
\draw [short] (23.25,14.25) -- (24.5,14.25);
\draw [short] (24.5,14.25) -- (24.5,14);
\draw [short] (24.5,14) -- (23.25,14);
\draw [short] (23.25,14) -- (23.25,14.25);
\draw [ dashed] (18,15.75) rectangle  (19,12.75);
\draw [ dashed] (20.5,15.75) rectangle  (21.5,12.75);
\draw [->, >=Stealth] (20.25,12.25) -- (20.25,11.25);
\node [font=\large] at (12.75,15.75) {$u$};
\draw [ dashed] (23,15.75) rectangle  (24,13.75);
\node [font=\LARGE] at (17.5,9.5) {$A$};
\draw  (18.25,10.5) rectangle (19,8);
\node [font=\LARGE] at (20,9.25) {=};
\draw  (20.5,10.5) rectangle (21.25,8);
\draw  (21.25,10.5) rectangle (21.5,8);
\node [font=\LARGE] at (22,9.25) {+};
\draw [short] (22.75,10.5) -- (23.5,10.5);
\draw [short] (22.75,10.5) -- (23.5,9.75);
\draw [short] (23.5,10.5) -- (23.5,9.75);
\draw [short] (22.75,9.75) -- (23.5,9.75);
\draw [short] (22.75,9.75) -- (22.75,9.5);
\draw [short] (22.75,9.5) -- (23.5,9.5);
\draw [short] (23.5,9.5) -- (23.5,9.75);
\draw [->, >=Stealth] (17.5,10.25) -- (15.75,10.25);
\node [font=\LARGE] at (8.25,9.25) {$A$};
\draw  (8.75,10.5) rectangle (10,8);
\node [font=\LARGE] at (10.75,9.25) {=};
\draw  (11.25,10.5) rectangle (12.5,8);
\draw  (12.5,10.5) rectangle (12.75,8);
\node [font=\LARGE] at (13,9.25) {+};
\draw [->, >=Stealth] (10.75,11.25) -- (10.75,12.25);
\draw [short] (13.75,10.5) -- (14.5,10.5);
\draw [short] (13.75,10.5) -- (14.5,9.75);
\draw [short] (14.5,10.5) -- (14.5,9.75);
\draw [short] (13.75,9.75) -- (14.5,9.75);
\draw [short] (13.75,9.75) -- (13.75,9.5);
\draw [short] (13.75,9.5) -- (14.5,9.5);
\draw [short] (14.5,9.5) -- (14.5,9.75);
\draw [short] (14.5,10.5) -- (15,10.5);
\draw [short] (15,10.5) -- (15,9.25);
\draw [short] (14.5,9.5) -- (14.75,9.25);
\draw [short] (14.75,9.25) -- (15,9.25);
\draw  (13.75,9.25) rectangle (15,9);

\node [font=\large] at (15.25,14) {$b$};
\node [font=\LARGE] at (9.5,14.25) {$U$};
\node [font=\LARGE] at (14.75,15) {$T$};
\node [font=\large] at (9.25,12) {Reduce to};
\node [font=\large] at (16,12.75) {Reorder Ritz values};
\node [font=\large] at (22.50,12) {Truncate by retaining };
\node [font=\large] at (16.75,10.75) {Extend with randomized Arnoldi};
\node [font=\large] at (22.50,11.5) {dotted sections};
\node [font=\large] at (9.25,11.5) {Schur form};
\node [font=\LARGE] at (12,14.25) {$U$};
\end{circuitikz}
}%
\caption{Visual representation of the randomized Krylov-Schur from \cref{alg:rKS}}
\label{fig:tikzrKS}
\end{figure}
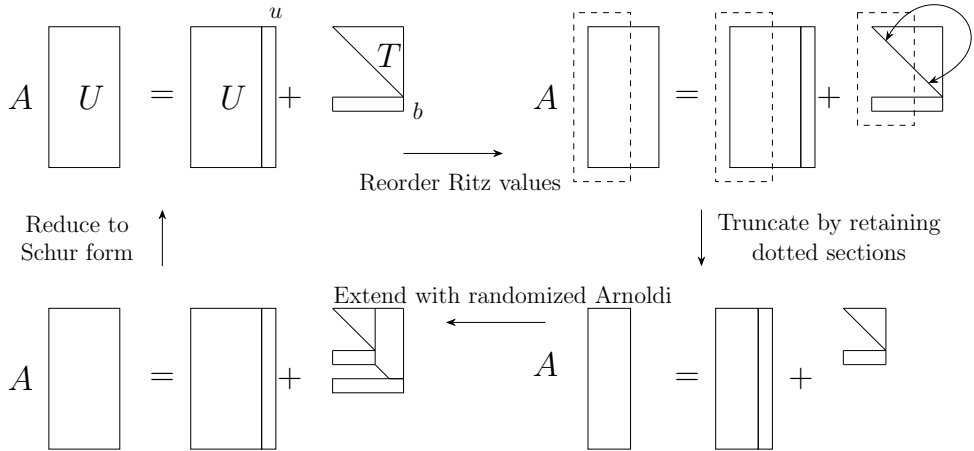

We now comment on two aspects of rKS: extending a length-$\sa$ Krylov-Schur decomposition to length $\ba$, and efficiently recovering sketched quantities during contraction.  First, the extension of a length-$\sa$ Krylov-Schur decomposition proceeds as follows. $(\U \; \ulast)$ is a sketch-orthonormal basis for $\K_{\sa+1}$, satisfying the rewriting of \cref{eq:rKS},
\begin{equation}
     \A \U = (\U \; \ulast) 
     \begin{pmatrix}
         \Tsa \\
         \blast\tsp
     \end{pmatrix}.
\end{equation}
 Then, $U_{\sa +1} = (\U \; \ulast)$, $S_{\sa +1} = (\Ssa \; \slast)$ and $B_{\sa+1} = 
        \begin{pmatrix}
         \Tsa \\
         \blast\tsp
        \end{pmatrix}$.
Next, $\A \ulast$ is computed and sketch-orthogonalized against $U_{\sa +1}$, and from there, a standard randomized Arnoldi procedure can be applied. This is detailed in \cref{alg:extend-rKS}. Note that \cref{alg:extend-rKS} reduces to a randomized Arnoldi procedure of length $\nshi$ when there is no initial decomposition, i.e.\@, when $\sa = 0$. In this case, given a starting vector $\tilde{u}_1 \in \Rn{n}$, simply initialize with $u_1 = \tilde{u}_1 / \norm{\Om \tilde{u}_1}$ and $b_1 = \norm{\Om \tilde{u}_1}$. This is used in \cref{alg:line:rKS-arnoldi-Init} of \cref{alg:rKS}.
\begin{algorithm}[!htbp]
    \caption{Extend a \rKS{} factorization}
    \label{alg:extend-rKS}
    \begin{algorithmic}[1]
        \Require A \rKS{} factorization $\A \U = \U \Bsa + \ulast \blast\tsp$ and the sketch $(\Ssa \; \slast) =  (\Om \U \; \Om \ulast)$.
        \Ensure A sketch-orthonormal Krylov decomposition $\A \Uba = \Uba \Bba + \uba \bba\tsp$ and the sketch $(\Sba \; \sba) = (\Om \Uba \; \Om \uba)$.
        \State $U_{\sa +1} = (\U \; \ulast)$ and $S_{\sa +1} = (\Ssa \; \slast)$ \label{alg:line:extendrKS-init}
        \State $B_{\sa+1} =      
        \begin{pmatrix}
         \Tsa \\
         \blast\tsp
        \end{pmatrix}$
        \For{$j = \sa+1, \dots, \ba$} 
        \State $w_j = \A u_j$  \label{alg:line:multAextend}  \Comment{Apply $\A$}
        \State $z_j = \Om w_j$ \Comment{Sketching (optional: sketch-orthogonalize against locked vectors, see \cref{sec:singlevecdef})} \label{alg:line:deflation}
        \State  $b_{j+1} = (\Om U_j)^\dag z_j = S_j^\dag z_j$ \Comment{Low dimensional projection step} \label{alg:line:lowprojstep}
        \State $\tilde{u}_{j+1} = w_j - U_j b_{j+1}  = w_j - U_j (\Om U_j)^\dag  \Om w_j = (I- P_\K^\Lspace)w_j$ \Comment{High dimension update}
        \State $\tilde{s}_{j+1} = \Om \tilde{u}_{j+1}$ and $\beta_{j+1} = \norm{\Om \tilde{u}_{j+1}}$ \Comment{Sketching step}
        \State $U_{j+1} = (U_j \; \uba = \tilde{u}_{j+1} / \beta_{j+1})$  \Comment{Augment basis}
        \State $S_{j+1} = (S_j \; \sba = \tilde{s}_{j+1} / \beta_{j+1})$  \Comment{Augment sketched basis} \label{alg:line:extendrKS-augmentS}
        \State $B_{j+1} =      
        \begin{pmatrix}
         B_j & b_{j+1}  \\
         0 \dots 0 & \beta_{j+1}
        \end{pmatrix}$  \Comment{Augment the Rayleigh quotient}
        \EndFor
    \end{algorithmic}
\end{algorithm}

Second, we discuss how to handle the sketch of the Krylov basis to avoid repeated sketching operations. In \cref{alg:extend-rKS}, the quantity $\Om U_j$ is written for clarity, but in practice it is recommended to store the sketch $S$ of the Krylov basis $U$, avoiding repeated sketching of $U_j$ at each iteration. This is not expensive in memory, since these are vectors in $\Rn{\dsk}$ with $\dsk \ll n$. This storage of $S$ is performed in \cref{alg:line:extendrKS-init} and \cref{alg:line:extendrKS-augmentS} of \cref{alg:extend-rKS}. To efficiently perform the contraction step on $S$ in rKS, note that from the factorization in \cref{alg:line:rKS-truncate} of \cref{alg:rKS},
\begin{equation}
    \tSsa \coloneqq \Om \tU = \Om \Uba Q_1 Q_2 (e_1, \dots, e_\sa) = \Sba  Q_1 Q_2 (e_1, \dots, e_\sa)
\end{equation}
and 
\begin{equation}
    \tslast \coloneqq \Om \tulast = \Om \uba = \sba.
\end{equation}
This explains the efficient truncations performed in \cref{alg:line:rKS-truncate} for $\tSsa$ and $\tslast$, instead of fully sketching $(\tU \; \tulast)$.

\subsection{Equivalence with randomized Implicitly Restarted Arnoldi}
\label{sec:equivKSIRA}
The following result establishes an equivalence between the randomized Implicitly Restarted Arnoldi method from \cite{Damas2024RandomizedImplicitlyRestarted} and the \rKS{} method.
\begin{theorem}
    Assume a randomized Arnoldi factorization as in \cref{def:rArno} and its equivalent \rKS{} factorization from \cref{th:KDeqRA}, both with the same $\epsilon$-embedding $\Om$ for $\Span{\V} = \Span{\U}$.
    \begin{align}
        \left\{
        \begin{aligned}
            \A \V & = \V \Hsa + \beta \ulast e_{\sa}^T \\       
            \A \U & = \U \Tsa + \ulast \blast\tsp.
        \end{aligned}
        \right.
    \end{align}
    Then, after a cycle of expansion and contraction steps using the rIRA and rKS methods, respectively, and assuming the retained Ritz values are distinct from the discarded ones, the resulting factorizations are equivalent in the sense that their spans are equal.
\end{theorem}
\begin{proof}
    Consider a randomized Arnoldi factorization of length $\sa$ and its equivalent \rKS{} decomposition after $\Hsa$ has been reduced to Schur form:
    \begin{align}
        \left\{
        \begin{aligned}
            \A \V & = \V \Hsa + \beta \ulast e_{\sa}^T \\       
            \A \U & = \U \Tsa + \ulast \blast\tsp
        \end{aligned}
        \right.
    \end{align}
    By \cref{th:KDeqRA}, $\Span{\V} = \Span{\U} = \K_\sa(\A, \V e_1)$. Since $\ulast$ is already sketch-orthogonal to $\V$, it remains unchanged in the equivalent \rKS{} decomposition. The expansion step using randomized Arnoldi yields bases $\Vba$ and $\Uba$ spanning the same subspace, namely $\K_{\ba}(\A, \V e_1)$, and both methods share the same last vector $\uba$.

    The rIRA method uses the implicit shifted QR method to discard the $\nshi$ unwanted Ritz values, while rKS reorders the Schur decomposition to move them to the lower-right block. In both cases, the reordering does not affect the last vector $\uba$. it modifies $\hat{\beta}$ to $\tilde{\beta}$ for rIRA and $\bba$ to $\tbba$ for rKS. Denote the unordered and ordered factorizations as:
    \begin{align}
        \left\{
        \begin{aligned}
            \A \Vba & = \Vba \Hba + \hat{\beta} \uba e_{\ba}^T \\       
            \A \Uba & = \Uba \Tba + \uba \bba\tsp
        \end{aligned}
        \right.
        \text{ and }
        \left\{
        \begin{aligned}
            \A \tVba & = \tVba \tHba + \tilde{\beta} \uba e_{\ba}^T \\       
            \A \tUba & = \tUba \tTba + \uba \tbba\tsp
        \end{aligned}
        \right.,
    \end{align}
    respectively. Our goal is to prove that $\Span{\tVba(e_1,\dots,e_\sa)} = \Span{\tUba(e_1,\dots,e_\sa)} $ so that the decompositions are equivalent in their range after the truncation step. Since $\Span{\Vba} = \Span{\Uba}$, there exists a non-singular matrix $W \in \Rnm{\ba}{\ba}$ such that 
    $$\Uba = \Vba W.$$
    This gives $(\Om \Uba)^T (\Om \Uba) = W^T (\Om \Vba)^T (\Om \Vba) W = W^T W$, but we also have that $(\Om \Uba)^T (\Om \Uba) = I_{\ba}$ using sketch-orthonormality of $\Uba$, thus $W$ is orthonormal. 
    Using the Rayleigh quotient expression of $\Tba$ from \cref{eq:rayleighquo-rKD} results in
    $$\Tba = (\Om \Uba)^T \Om \A \Uba = W^T (\Om \Vba)^T \Om \A \Vba W = W^T \Hba.W. $$
    This shows that $\Tba$ and $\Hba$ are similar, hence they discard and retain the same Ritz values in their respective reordering processes. We note $Q$ and $R$ the unitary similarity transforms for the reordering processes of rIRA and rKS, respectively. By definition, $\Tba \coloneqq R \tTba R^T$ and thus $R(e_1,\dots,e_\sa)$ spans the Schur vectors of $\Tba$ associated to the retained Ritz values. As stated in \cite[Theorem 3.1]{Stewart2002KrylovSchurAlgorithm}, $Q(e_1,\dots,e_\sa)$ spans the Schur vectors of $\Hba$ associated to the same Ritz values, and given that $\Tba = W^T \Hba W$ and that these eigenspaces are simple, we have that $W^T Q(e_1,\dots,e_\sa)$ has the same span as $R(e_1,\dots,e_\sa)$. Thus there exists a non-singular matrix $Z$ such that $R(e_1,\dots,e_\sa) = W^T Q(e_1,\dots,e_\sa) Z$. Ultimately, using $\tUba \coloneqq \Uba R$ and $\tVba \coloneqq \Vba Q $ gives
    $$\tUba(e_1,\dots,e_\sa) = \Uba R(e_1,\dots,e_\sa)  = \Vba W R(e_1,\dots,e_\sa)  = \Vba W W^T Q(e_1,\dots,e_\sa) Z = \tVba(e_1,\dots,e_\sa) Z,$$
    which shows $\Span{\tVba(e_1,\dots,e_\sa)} = \Span{\tUba(e_1,\dots,e_\sa)}$.
\end{proof}
This equivalence allows results from rIRA to be transferred to rKS. In particular, rIRA (and thus rKS) acts as a polynomial filtering method on the eigenvalue problem, modifying the starting vector of the Krylov subspace by $\uinitj = \psi(A) \uinit$ at each restart, where $\psi(A)$ is a polynomial that damps unwanted eigen-directions in $\A$, see \cite{Lehoucq1995DeflationTechniquesImplicitly,Damas2024RandomizedImplicitlyRestarted} for more details.

\subsection{A structure-preserving re-orthogonalization process}
The critical step of the rKS method is the orthogonalization process, whose stability is required to obtain accurate Ritz eigenpairs. It occurs during the randomized Arnoldi method (see \cref{alg:line:extend-facto} of \cref{alg:rKS}), and more precisely at the projection step (\cref{alg:line:lowprojstep} of \cref{alg:extend-rKS}). Stability is measured by how well-conditioned the sketched basis $S$ is, and consequently $U$. Ideally, $S$ is exactly orthonormal so that $\kappa(S) = 1$, which by \cref{cor:embeddsvals} gives $\kappa(U) = O(1+\epsilon)$. However, in finite precision arithmetic, algorithms such as Gram-Schmidt suffer from a loss of orthogonality, where $\kappa(S)$ grows slowly and can eventually become large; see \cite{Giraud2005lossorthogonalityGram}.

We introduce an approach to re-orthogonalize a sketch-orthonormal basis. It relies on the \emph{whitening} strategy from \cite{Nakatsukasa2024FastAccurateRandomized}, which is related to the sketch and precondition approach from \cite{Rokhlinfastrandomizedalgorithm2008}. Consider a sketch-orthonormal Krylov decomposition $\A \U = \U \Bsa + \ulast \blast\tsp$ as in \cref{def:sketchorthKD}. Let $\Ssa$ be the sketch of $\U$. The idea is to compute a stable QR factorization of $\Ssa$ using Householder transformations and obtain $\Ssa = QR$. The stability of Householder QR is discussed in \cite{Higham2002Accuracystabilitynumerical}, where it is shown that $\kappa(Q) \leq \sa^{3/2} \dsk \mathbf{u}$, with $\mathbf{u}$ the unit round-off. This is sufficiently small and independent of the large dimension $n$. We can then use $R$ as a non-singular similarity transformation for the Krylov decomposition,
\begin{equation}
    \A \U R^{-1} = \U  R^{-1} R \Bsa R^{-1} + \ulast \blast\tsp R^{-1} \implies \A \tU = \tU \tBsa + \ulast \tblast\tsp.
\end{equation}
This is a sketch-orthonormal Krylov decomposition. The new basis $\tU = \U R^{-1}$ has sketch $\tSsa = \Om \tU = \Om \U R^{-1} = \Ssa R^{-1} = Q$, and is thus orthonormal up to machine precision.

This method is inexpensive, as the QR factorization of $\Ssa \in \Rnm{\dsk}{\sa}$ is affordable due to the small sketch dimension, and the update $\tU = \U R^{-1}$ can be performed via backward triangular solve, since $R$ is upper-triangular. Hence, this re-orthogonalization process can be easily integrated in a Krylov decomposition. Moreover, it preserves the randomized Arnoldi structure. Consider $\A V_\sa = V_\sa \Hsa + \beta v_{\sa+1} e_\sa^T$ as in \cref{def:rArno}. The same procedure applies. With $\Ssa = \Om V_\sa$, compute $\Ssa = QR$ and obtain
\begin{equation}
    \A  V_\sa R^{-1} =  V_\sa R^{-1} R \Hsa R^{-1} + \beta v_{\sa+1} e_\sa^T R^{-1}.
\end{equation}
Note that $R^{-1}$ is upper-triangular, so its last row is of the form $(0, 0, \dots, 0 , \alpha)$, giving $e_\sa^T R^{-1} = \alpha e_\sa^T$. Defining $\tilde{\beta} = \alpha \beta$, we obtain a randomized Arnoldi factorization with a better conditioned Krylov basis.  This shows that, subject to a stable QR factorization of the Krylov basis sketch, we can re-orthogonalize the full basis in the Krylov decomposition or Arnoldi framework while preserving the structure in both cases.

\section{Ritz pairs convergence and associated deflation technique}
\label{sec:deflation}
This section explains how to monitor the convergence of Ritz pairs in the rKS method, and describes a deflation procedure for converged pairs.

\subsection{Monitoring the residual errors}
\label{sec:errormonitoring}
This subsection focuses on computing the relative residual error $\norm{\A \rivec - \rival \rivec} / \norm{\A \rivec}$ for an approximate eigenpair $(\rival, \rivec)$ obtained at \cref{alg:line:rksreserrors} in \cref{alg:rKS}. Note that $\rivec$ is typically not of unit norm, since $\rivec = \U \yrivec$ and $\U$ is sketch-orthonormal rather than $\ell_2$-orthonormal. However, the relative residual error is invariant under normalization by $\norm{\rivec}$, so normalization can be ignored.

The starting point is the strategy used in \rArno{} in \cite{Nakatsukasa2024FastAccurateRandomized}, \cite{Damas2024RandomizedImplicitlyRestarted}. Given a \rArno{} factorization $\A \V = \V \Hsa + \beta v_{\sa+1} e_\sa^T$ and a Ritz pair $(\rival,\rivec = \V \yrivec)$, it holds that 
\begin{equation}
    \label{eq:rArno-reserror}
    \sqrt{\frac{1}{1+\epsilon}} \frac{\beta \abs{e_\sa^T\yrivec}}{\norm{\A \rivec}} 
    \leq \frac{\norm{\A \rivec - \rival \rivec}}{\norm{\A \rivec}} 
    \leq \sqrt{\frac{1}{1-\epsilon}} \frac{\beta \abs{e_\sa^T\yrivec}}{\norm{\A \rivec}}.
\end{equation}
This is obtained by multiplying the \rArno{} relation \cref{eq:rArno} by an eigenvector $\yrivec$ of the Hessenberg matrix. Now, consider a \rKS{} factorization from \cref{def:rKS}
\begin{equation}
\label{eq:rKS-for-monitoring}
    \A \U = \U \Tsa + \ulast \blast\tsp.
\end{equation}
If $(\rival,\rivec = \U \yrivec)$ is a Ritz pair, i.e.\@, $(\rival, \yrivec)$ is an eigenpair of $\Tsa$, then multiplying \cref{eq:rKS-for-monitoring} by $\yrivec$ gives
\begin{equation}
    \label{eq:RitzError}
    \A \rivec - \rival \rivec = \ulast \blast\tsp \yrivec.
\end{equation}
The relative error can be obtained by scaling the error by $\norm{\A \rivec}$. While this is feasible, it is costly in practice, as it requires a matrix-vector product with $\A$ and the computation of $\U \yrivec$ for each Ritz pair. The approximation $\norm{\A \rivec} \approx \norm{\rival \rivec} = \abs{\rival} \norm{\rivec}$ can be used instead, which is most accurate near convergence. Consequently, the residual error can be computed as
\begin{equation}
    \label{eq:exact_rks_reserror}
    \frac{\norm{\A \rivec - \rival \rivec}}{\norm{\A \rivec}} = \frac{\norm{\ulast} \abs{\blast\tsp \yrivec}}{\norm{\A \rivec}} \approx \frac{ \norm{\ulast}}{ \norm{\rivec}} \frac{ \abs{\blast\tsp \yrivec}}{ \abs{\rival}}. 
\end{equation}
Still, \cref{eq:exact_rks_reserror} requires computing two norms in $\Rn{n}$, which may be expensive. Instead, the $\epsilon$-embedding property can be used to derive a bound similar to \cref{eq:rArno-reserror} for rKS. Using
\begin{align}
    \sqrt{\frac{1}{1+\epsilon}} \norm{\Om (\A \rivec - \rival \rivec)} & \leq \norm{\A \rivec - \rival \rivec} \leq \sqrt{\frac{1}{1-\epsilon}} \norm{\Om (\A \rivec - \rival \rivec)}  \\
    \sqrt{1-\epsilon} \frac{1}{\norm{\Om (\A \rivec)}}& \leq \frac{1}{\norm{\A \rivec}}   \leq \sqrt{1+\epsilon} \frac{1 }{\norm{\Om (\A \rivec)}}   
\end{align}
from \cref{eq:epsembedd} and
\begin{align}
    \norm{\Om (\A \rivec - \rival \rivec)} & = \norm{\Om \ulast \blast\tsp \yrivec} = \abs{\blast\tsp \yrivec} & \text{since the sketch $\ulast$ is of unit norm} \\
    \norm{\Om (\A \rivec)} &  \approx \norm{ \Om(\rival \U \yrivec)} = \abs{\rival} & \text{since $\U$ is sketch-orthonormal and $\yrivec$ is of unit norm}
\end{align}
we obtain the bound
\begin{equation}
    \label{eq:rks_reserror_bound}
        \sqrt{\frac{1-\epsilon}{1+\epsilon}} \frac{\abs{\blast\tsp \yrivec}}{\abs{\rival}} 
        \lessapprox 
        \frac{\norm{\A \rivec - \rival \rivec}}{\norm{\A \rivec}}  
        \lessapprox 
        \sqrt{\frac{1+\epsilon}{1-\epsilon}}  \frac{\abs{\blast\tsp \yrivec}}{\abs{\rival} } .
\end{equation}
The main advantage of the bound in \cref{eq:rks_reserror_bound} is that the quantity $\abs{\blast\tsp \yrivec} / \abs{\rival}$ is directly available at negligible cost during the rKS algorithm, allowing for a simple estimate of the relative residual error up to a small distortion factor $ \sqrt{\frac{1 \mp \epsilon}{1 \pm \epsilon}}$.
\begin{remark}
    If one prefers not to use the approximation $\norm{\A \rivec} \approx \norm{\rival \rivec}$, the following exact bound can be used:
    \begin{equation}
        \sqrt{\frac{1-\epsilon}{1+\epsilon}} \abs{\blast\tsp \yrivec} \leq \frac{\norm{\A \rivec - \rival \rivec}}{\norm{\rivec}}  \leq \sqrt{\frac{1+\epsilon}{1-\epsilon}} \abs{\blast\tsp \yrivec},
    \end{equation}
    which is obtained similarly. It uses the normalized Ritz vector $\rivec / \norm{\rivec}$, but it is not relative to the eigenvalue $\rival$.
\end{remark}

\subsection{Single vector deflation}
\label{sec:singlevecdef}
We now address deflation of converged eigenpairs. If a Ritz pair $(\rival, \rivec)$ has a small residual error, we want it to be excluded from further iterations. In rKS, deflation is closely related to partial sketch-orthonormal Schur factorizations of $\A$, introduced below. Let $\A Q = Q T$ be the real Schur decomposition of $\A$ as in \cref{prop:realSchur}. For any $k \in \{1,\dots,n\}$, it holds that
\begin{equation}
   \A q_\sa = \eival_\sa q_\sa + \sum_{i=1}^{\sa-1} t_{i,k} q_i.
\end{equation}
Thus, $ \Span{q_1,\dots,q_\sa} $ forms an invariant subspace of $ \A $. This allows for the construction of a $\sa$-partial Schur factorization $ AQ_\sa = Q_\sa \Tsa $, where $ Q_\sa \in \Rnm{n}{\sa} $ is orthonormal, and $ \Tsa \in \Rnm{\sa}{\sa} $ is block upper triangular. For any subset of eigenvalues $ \{\eival_1,\dots,\eival_\sa \} $ of $\A$ closed under complex conjugation, this factorization can be achieved by taking $Q_\sa$ as the first $\sa$ Schur vectors from \cref{eq:realSchur}, with these vectors chosen to place $ \{\eival_1,\dots,\eival_\sa \} $ in the top-left section of the $ n \times n $ upper triangular matrix $T$. The columns of $Q_\sa$ can then be sketch-orthonormalized using, for example, the RGS method, yielding $Q_\sa = \U R$ with $\U$ sketch-orthonormal. Assuming no zero eigenvalues among $\{\eival_1,\dots,\eival_\sa \}$, $R$ is non-singular. Then
\begin{equation}
    \label{eq:existence_of_sketchschur}
    \A Q_\sa = Q_\sa \Tsa \implies \A \U = \U (R \Tsa R^{-1}) = \U \tTsa,
\end{equation}
where $\tTsa$ is similar to $\Tsa$, so they share the same eigenvalues. Moreover, $\tTsa$ is upper-triangular as a product of upper-triangular matrices. The existence of the decomposition in \cref{eq:existence_of_sketchschur} motivates the following definition.
\begin{definition}
    \label{def:kpartialsketchSchur}
    A $\sa$-partial sketch-orthonormal real Schur factorization of $ \A $ is defined as
    \begin{equation}
      \A \U = \U \Tsa,
    \end{equation}
    where $ \U \in \Rnm{n}{\sa} $ is sketch-orthonormal and $ \Tsa $ is block upper triangular, with blocks of dimension up to $ 2 \times 2 $. The columns of $ \U $ are called sketch-orthonormal Schur vectors of $ \A $, in analogy to Schur vectors. They can be chosen so that the eigenvalues of $ \Tsa$ are any subset $ \{\eival_1,\dots,\eival_\sa \} $ of eigenvalues of $\A$.
\end{definition}
We now establish the connection between partial sketch-orthonormal Schur factorizations of $\A$ and converged Ritz pairs in rKS. Assume the following \rKS{} factorization holds,
\begin{equation}
    \label{eq:deflatedAu1U2}
    \A (u_1 \; U_2) = (u_1 \; U_2)
    \begin{pmatrix}
        \rival_1 & T_{12} \\
        0 & T_{22}
    \end{pmatrix}
    + u (b_1\tsp \; b_2\tsp),
\end{equation}
where $b_1$ is a scalar. Given the upper-triangular structure of $T$, this means that
\begin{equation}
    \frac{\norm{\A u_1 - \rival_1 u_1}}{\norm{u_1}} = \frac{\abs{b_1}}{\norm{u_1}}.
\end{equation}
In other words, if $b_1 = 0$, the normalized first column of $U$ acts as an eigenvector and a sketch-orthonormal Schur vector for $\A$, with corresponding eigenvalue $\rival_1$ (the top-left entry of $T$). In practice, we require $\abs{b_1}$ to be smaller than a given tolerance $\eta$, and the error incurred by setting it to zero is discussed in \cref{sec:perturbed_problem}. For now, we assume $b_1 = 0$, allowing the algorithm to continue with the deflated decomposition,
\begin{equation}
    \label{eq:deflatedAU2}
    \A U_2 = U_2 T_{22} + u b_2\tsp.
\end{equation}
It is important to note that, for this new decomposition to hold and to find other eigenvalues of $\A$, each new vector produced during the expansion step must be sketch-orthogonalized against $u_1$. That is, \cref{alg:line:multAextend} of \cref{alg:extend-rKS} becomes
\begin{equation}
    w_j =  A u_j - u_1 (\Om u_1)^\dag \Om A u_j,
\end{equation}
so that $w_j \Omperp u_1$. Note that $(\Om u_1)^\dag = (\Om u_1)^T$. This naturally leads to the deflated operator $A_1 \coloneqq (I - u_1 (\Om u_1)^T \Om)A$, so the newly computed vector $w_j$ in the Krylov subspace satisfies $w_j = A_1 u_j$. We denote $P_{u_1}^\Om \coloneqq u_1 (\Om u_1)^\dag \Om$ as the oblique projector on $\Span{u_1}$, with $A_1 = (I-P_{u_1}^\Om)A$. We show that $A_1$ shares the same sketch-orthonormal Schur vectors as $\A$.
\begin{proposition}
    \label{prop:A1def}
    Let $\A$ be a square matrix and consider a $\sa$-partial sketch-orthonormal real Schur factorization $\A \U = \U \Tsa$ corresponding to the subset of eigenvalues $(\eival_1, \dots, \eival_\sa)$, with $u_1 = \U e_1$ and $t_{11} = \lambda_1$. Note $P_{u_1}^\Om \coloneqq u_1 (\Om u_1)^\dag \Om$ the oblique projector onto $\Span{u_1}$. Then the operator 
    \begin{equation}
      A_1 \coloneqq (I - P_{u_1}^\Om)A  
    \end{equation}
    shares the same $\sa$-partial sketch-orthonormal Schur vectors with $\A$, and has the subset of eigenvalues $(0, \eival_2,\dots,\eival_\sa)$.  Moreover, the eigenvectors $\tilde{u}_i$ of $A_1$ are sketch-orthogonal to $u_1$, the subspace $\Span{u_1, \tilde{u}_2, \dots, \tilde{u}_\sa}$ is invariant under $\A$  and a pair $(u_1, \tilde{u}_i)$ forms a pair of sketch-orthonormal Schur vectors for $\A$ for all $i \in [2,\sa]$.
\end{proposition}
\begin{proof}
    We have the following computations
    \begin{align*}
        A_1 \U & = (I - u_1 (\Om u_1)^T \Om)A \U = (I - u_1 (\Om u_1)^T \Om) \U \Tsa \\
        & = \U \Tsa - u_1 (\Om u_1)^T (\Om \U) \Tsa =  \U \Tsa - u_1 e_1^T \Tsa \\
        & = \U \Tsa - \U e_1 e_1^T \Tsa = \U (\Tsa - \begin{pmatrix}
            \lambda_1 & 0 \\
            0 & 0 
        \end{pmatrix}) \\
        & = \U \tTsa
    \end{align*}
    where $\tTsa$ equals $\Tsa$ except it has a zero on its top-left entry. The decomposition $A_1 \U = \U \tTsa$ is a $\sa$-partial sketch-orthonormal real Schur factorization for $A_1$, with the same sketch-orthonormal Schur vectors $\U$ as $\A$. Now, let $(\tilde{x}_i,\lambda_i)$ be an eigenpair for $A_1$ with $\lambda_i \neq 0$. Then
    \begin{align*}
        A_1 \tilde{x}_i & =  (I - u_1 (\Om u_1)^T \Om) \A \tilde{x}_i  = \A \tilde{x}_i - u_1 (\Om u_1)^T \Om \A \tilde{x}_i = \A \tilde{x}_i  - u_1 \alpha,
    \end{align*}
    with $\alpha \coloneqq (\Om u_1)^T \Om \A \tilde{x}_i$. Since $A_1 \tilde{x}_i = \lambda_i \tilde{x}_i$ by definition, this gives $\A \tilde{x}_i  = \lambda_i \tilde{x}_i + \alpha u_1 $,  such that $\A \tilde{x}_i \in \Span{u_1,\tilde{x}_i}$. In addition, $\tilde{x}_i$ and $u_1$ are sketch-orthogonal:
    \begin{align*}
        (\Om u_1)^T \Om  \tilde{x}_i & = \frac{1}{\lambda_i} (\Om u_1)^T (\Om \A \tilde{x}_i - \Om u_1 (\Om u_1)^T \Om \A \tilde{x}_i) \\
        & = \frac{1}{\lambda_i}  ((\Om u_1)^T \Om \A \tilde{x}_i - (\Om u_1)^T \Om \A \tilde{x}_i) = 0
    \end{align*}
    using $ (\Om u_1)^T  (\Om u_1) = 1$. As such, $(u_1,\tilde{x}_i)$ forms a pair of sketch-orthonormal Schur vectors for $\A$, and the subspace $\Span{u_1, \tilde{u}_2, \dots, \tilde{u}_\sa}$ is invariant under $\A$.
\end{proof}
This shows that the single vector deflation procedure preserves the sketch-orthonormal Schur vectors of $\A$. Finally, we show that implicitly using the deflated operator $A_1$ by explicitly sketch-orthogonalizing new vectors against $u_1$ justifies the deflated decomposition in \cref{eq:deflatedAU2}. Left-multiplying \cref{eq:deflatedAu1U2} by $(I -P_{u_1}^\Om)$ yields
\begin{equation*}
    (I -P_{u_1}^\Om) \A (u_1 \; U_2) = (I -P_{u_1}^\Om) (u_1 \; U_2)
    \begin{pmatrix}
        \rival_1 & T_{12} \\
        0 & T_{22}
    \end{pmatrix}
    + (I -P_{u_1}^\Om) u (b_1\tsp \; b_2\tsp).
\end{equation*}
The terms reduce as follows, assuming $\abs{b_1} = 0$ and thus $\A u_1 = \lambda_1 u_1$:
\begin{itemize}
    \item $(I -P_{u_1}^\Om) \A u_1 = \lambda_1 u_1 - \lambda_1 u_1 = 0$, since the projection of $u_1$ by $P_{u_1}^\Om$ is $u_1$ itself.
    \item $(I -P_{u_1}^\Om) \A U_2 = A_1 U_2$.
    \item $(I -P_{u_1}^\Om) u_1 = 0$.
    \item $(I -P_{u_1}^\Om) U_2 = U_2 -  u_1 (\Om u_1)^T \Om U_2 = U_2$, since $u_1$ is sketch-orthogonal to $U_2$ by construction.
    \item $(I -P_{u_1}^\Om) u = u -  u_1 (\Om u_1)^T \Om u = u $, since $u$ is sketch-orthogonal to $u_1$ by construction.
\end{itemize}
This results in the decomposition
\begin{equation}
    A_1 (0 \; U_2) = (0 \; U_2)
    \begin{pmatrix}
        \rival_1 & T_{12} \\
        0 & T_{22}
    \end{pmatrix}
    + u (0 \; b_2\tsp) \implies A_1 U_2 = U_2 T_{22} + u b_2\tsp.
\end{equation}
The single deflation procedure then consists of proceeding implicitly with the last equation, as proposed in \cref{eq:deflatedAU2}. This strategy is often called \emph{locking} in the context of restarted Arnoldi methods; see \cite{Lehoucq1998ARPACKusersguide,Lehoucq1995DeflationTechniquesImplicitly,saad2011numerical,Stewart2002KrylovSchurAlgorithm,Kressner2005NumericalMethodsGeneral} for more details in the deterministic setting.

\subsection{Subspace deflation}
We now address the deflation of multiple converged vectors and their associated subspace, generalizing \cref{sec:singlevecdef}. Suppose the single vector $u_1$ has been deflated as described above, and now
\begin{equation}
    \label{eq:deflatedA1u2U3}
    A_1 (u_2 \; U_3) = (u_2 \; U_3)
    \begin{pmatrix}
        \rival_2 & T_{23} \\
        0 & T_{33}
    \end{pmatrix}
    + u (b_2\tsp \; b_3\tsp).
\end{equation}
If $\abs{b_2} = 0$, the single vector deflation procedure can be applied, resulting in the operator
\begin{equation}
    A_2 \coloneqq (I - P_{u_2}^\Om) A_1 = (I - P_{u_2}^\Om) (I - P_{u_1}^\Om) \A.
\end{equation}
In this situation, by \cref{prop:A1def}, $(u_1, u_2)$ forms a pair of sketch-orthonormal Schur vectors for $\A$, since $u_2$ is an eigenvector of $A_1$. We could then lock $u_2$ as for $u_1$ and work with $A_2 U_3 = U_3 T_{33} + u b_3\tsp$. This would deflate the subspace $\Span{u_1,u_2}$, corresponding to the eigenspace $\Span{\eivec_1,\eivec_2}$, where $\eivec_1,\eivec_2$ are eigenvectors of $\A$ associated with the simple eigenvalues $\lambda_1,\lambda_2$.

An important property is that $(I - P_{u_2}^\Om) (I - P_{u_1}^\Om) = (I - P_{U_{12}}^\Om)$, where $U_{12} \coloneqq (u_1 \; u_2)$. Indeed,
\begin{align}
    (I - P_{u_2}^\Om) (I - P_{u_1}^\Om) = I - P_{u_2}^\Om - P_{u_1}^\Om + P_{u_2}^\Om P_{u_1}^\Om = I - P_{u_2}^\Om - P_{u_1}^\Om 
\end{align}
since $u_1 \Omperp u_2$, so $P_{u_2}^\Om P_{u_1}^\Om = 0$, and:
\begin{align}
    (I - P_{U_{12}}^\Om) & = I - (u_1 \; u_2) (\Om (u_1 \; u_2))^T \Om  =  I - (u_1 \; u_2)
    \begin{pmatrix}
        (\Om u_1)^T \Om \\
        (\Om u_2)^T \Om
    \end{pmatrix} \nonumber \\
    & = I - u_1 (\Om u_1)^T \Om - u_2 (\Om u_2)^T \Om =  I - P_{u_2}^\Om - P_{u_1}^\Om.
\end{align}
This can be generalized to $q$ vectors $(u_1, \dots, u_q)$, showing that deflating many vectors sequentially is equivalent to deflating them all at once. The following result summarizes the deflation strategy for rKS.
\begin{theorem}
    \label{th:exactSubspaceDeflation}
    Let $\A$ be a square matrix and assume a \rKS{} factorization $\A \U = \U \Tsa + \ulast \blast^*$ holds, as in \cref{def:rKS}. Let $q$ be an integer smaller than $\sa$ and partition the decomposition as 
    $    A (U_q \; \tilde{U}) = (U_q \; \tilde{U})
    \begin{pmatrix}
        T_{qq} & * \\
        0 & \tilde{T}
    \end{pmatrix}
    + u (b_q\tsp \; \tilde{b}\tsp)$, where the spectrum of $T_{qq}$ is $(\lambda_1,\dots,\lambda_q)$. 
    If the first $q$ components of $\blast$ are zeros, i.e.\@\@ $b_q\tsp = (0, \dots, 0)$, then the sketch-orthonormal set of vectors $U_q$ can be deflated using the decomposition
    \begin{equation}
        \label{eq:AqSchurDeflatedSubspace}
        (I - P_{U_q}^\Om) A \tilde{U} = \tilde{U} \tilde{T} + u \tilde{b}\tsp
    \end{equation}
    where $P_{U_q}^\Om \coloneqq U_q (\Om U_q)^\dag \Om$ is an oblique projector onto $\Span{U_q}$. Moreover
    \begin{itemize} 
        \item $(\lambda_1,\dots,\lambda_q)$ are eigenvalues of $\A$ and columns of $U_q$ are sketch-orthonormal Schur vectors of $\A$ spanning the eigenspace associated to these eigenvalues.
        \item Let $(\lambda_1,\dots,\lambda_q, \lambda_{q+1}, \dots, \lambda_\sa)$ by any subset of eigenvalues of $\A$ containing $(\lambda_1,\dots,\lambda_q)$. This subset has an associated $\sa$-partial sketch-orthonormal real Schur factorization for $\A$. Then, the operator $A_q \coloneqq (I - P_{U_q}^\Om) A$ shares the same sketch-orthonormal Schur vectors as in that factorization, only now they are associated to eigenvalues $(0,\dots,0, \lambda_{q+1}, \dots, \lambda_\sa)$. 
        \item Reciprocally, if $U_{r} \in \Rnm{n}{r}$ are $r \; (\leq k-q)$ sketch-orthonormal Schur vectors for $A_q$ associated with nonzero eigenvalues, then the columns of $(U_q \; U_r)$ are sketch-orthonormal Schur vectors for $\A$, and as such they are invariant under $\A$. 
    \end{itemize}
\end{theorem}
\begin{proof}
First, the fact that $b_q$ has all zero components directly gives $A U_q = U_q T_{qq}$ by truncating the partition of the \rKS{} decomposition to its first $q$ columns. Thus, columns of $U_q$ are sketch-orthonormal vectors for $\A$ associated with the eigenvalues $(\eival_1, \dots, \eival_q)$ by definition of a sketch-orthonormal Schur factorization of length $q$.

Next, let $(\lambda_{q+1},\dots,\lambda_{\sa})$ be any $k-q$ eigenvalues of $\A$ with $U_{k-q}$ the associated sketch-orthonormal Schur vectors, i.e.\@, $\A U_{k-q} = U_{k-q} T_{k-q,k-q}$. Then, $\A (U_q \; U_{k-q}) = (U_q \; U_{k-q})T = UT$, where $U \coloneqq (U_q \; U_{k-q})$ and $T \coloneqq \begin{pmatrix}
    T_{qq} & * \\
    0 & T_{k-q,k-q}
\end{pmatrix}$. Similarly to the proof of \cref{prop:A1def}, it holds that:
\begin{align*}
    \A_q U & = (I- U_q(\Om U_q)^\dag \Om) A U = (I- U_q(\Om U_q)^\dag \Om) U T = UT - U_q(\Om U_q)^\dag \Om U T \\
    & = UT - U_q(\Om U_q)^T \Om (U_q \; U_{k-q}) T = UT - U_q(I_q \; 0)T = UT - U \begin{pmatrix}
        I_q \\
        0
    \end{pmatrix}(I_q \; 0)T \\ 
    & = UT - \begin{pmatrix}
    I_q & 0 \\
    0 & 0
\end{pmatrix} \begin{pmatrix}
    T_{qq} & * \\
    0 & T_{k-q,k-q}
\end{pmatrix} 
 = U(T - \begin{pmatrix}
    T_{qq} & * \\
    0 & 0
\end{pmatrix} )= U \begin{pmatrix}
    0 & 0 \\
    0 & T_{k-q,k-q}
\end{pmatrix}.
\end{align*}
As such, columns of $U$ are sketch-orthonormal Schur vectors for $\A_q$, associated with eigenvalues $(0,\dots,0,\eival_{q+1},\dots,\eival_\sa)$. The goal of zeroing converged eigenvalues of $\A$ in $\A_q$ is thus achieved.

Finally, suppose $\A_q U_r = U_r T_r$, i.e.\@, we have found $r \leq k-q$ converged sketch-orthonormal Schur vectors for $\A_q$ associated with nonzero eigenvalues, so $T_r$ is non-singular. We show that $U \coloneqq (U_q \; U_r)$ is invariant under $\A$. First, $\A_q U_r = U_r T_r \implies \A U_r = U_r + P_{U_q}^\Om \A U_r \in \Span{U_r \; U_q}$. Since $U_q$ is already invariant under $\A$, it follows that $\A (U_q \; U_r) \in \Span{U_q \; U_r}$. Now, $U_r$ and $U_q$ are sketch-orthogonal:
\begin{align*}
    (\Om U_q)^T (\Om U_r)&  = (\Om U_q)^T (\Om \A_q U_r T_r^{-1}) = (\Om U_q)^T (\Om (I - U_q (\Om U_q)^T \Om)) \A U_r T_r^{-1} \\
    & = (\Om U_q)^T (\Om  - \Om  U_q (\Om U_q)^T \Om)) \A U_r T_r^{-1} = ((\Om U_q)^T \Om - (\Om U_q)^T \Om) \A U_r T_r^{-1} = 0,
\end{align*}
using sketch-orthonormality of $U_q$. Consequently, columns of $U = (U_q \; U_r)$ are sketch-orthonormal Schur vectors of $\A$.
\end{proof}

This theorem shows that sketch-orthonormal Schur vectors of $\A$ can be deflated whenever their associated error is zero, and as a result, their corresponding converged eigenvalues will not be part of subsequent computations. These eigenvalues are now set to zero in $A_q$. Moreover, any sketch-orthonormal Schur vectors that later converge for $A_q$ are still sketch-orthonormal Schur vectors for $\A$. This reduces the number of targeted eigenpairs each time locking occurs, and may enhance convergence by changing the gap ratios between eigenvalues (e.g., if $\lambda_1$ was the largest modulus and now $\lambda_2$ is). Note that any \rKS{} factorization such as \cref{eq:deflatedAu1U2} can be reordered to test different vectors $u_{q+1}$ and their associated error $b_{q+1}$, a strategy known as \emph{aggressive deflation} in \cite{Kressner2005NumericalMethodsGeneral}. In practice, as in single-vector deflation, working with $A_q$ is done by sketch-orthogonalizing $A u_j$ against $U_q$ in \cref{alg:line:multAextend} of \cref{alg:extend-rKS}:
\begin{equation}
    w_j =  A u_j - U_q (\Om U_q)^\dag \Om A u_j.
\end{equation}

\subsection{The perturbed deflated problem}
\label{sec:perturbed_problem}
In practice, it is not feasible to require the components of $b$ to be exactly zero. Instead, a convergence criterion $\eta$ is chosen so that deflation occurs when the first $q$ entries of $\abs{b}$ are smaller than $\eta$. We present a bound for the resulting error on a Ritz pair using this criterion. Consider a partitioned \rKS{} factorization of size $\sa$,
\begin{equation}
    \label{eq:AUqU}
    \A (U_q \; \tilde{U}) = (U_q \; \tilde{U})
    \begin{pmatrix}
        T_{qq} & * \\
        0 & \tilde{T}
    \end{pmatrix}
    + \ulast (b\tsp \; \tilde{b}\tsp).
\end{equation}
We assume all $q$ components of $b$ are smaller in absolute value than $\eta$, i.e.\@, $\abs{b_i} \leq \eta$ for $i = 1,\dots,q$, so $\norm{b} \leq \sqrt{q} \eta$. If we take an eigenpair of $T_{qq}$ satisfying $T_{qq} \yrivec = \rival \yrivec$, then from \cref{eq:RitzError},
\begin{equation*}
    \A (U_q \yrivec) = \rival (U_q \yrivec) + \ulast b\tsp \yrivec
\end{equation*}
such that
\begin{equation}
    \label{eq:etaControlRitz}
    \norm{\A (U_q \yrivec) - \rival (U_q \yrivec)} = \norm{\ulast b\tsp \yrivec} \leq \frac{\sqrt{q}}{\sqrt{1-\epsilon}} \eta
\end{equation}
using the $\epsilon$-embedding property for $\norm{\ulast}$, the Cauchy–Schwarz inequality for $\abs{b\tsp \yrivec} \leq \norm{b} \norm{\yrivec}$, and $\norm{\yrivec} = 1$ as an eigenvector. This shows that the threshold $\eta$ for $\abs{b}$ actually controls the error on the Ritz pairs associated with the Ritz values of $T_{qq}$, since the factor $\frac{\sqrt{q}}{\sqrt{1-\epsilon}}$ is $O(1)$ in practice. The process of controlling Ritz pair errors through the vector $b_q$ is mentioned as a more restrictive convergence criterion based on Schur vectors in the deterministic case in \cite[Chapter 3]{Kressner2005NumericalMethodsGeneral}.
\begin{remark}
    \Cref{eq:etaControlRitz} can be slightly improved by noting that for an upper-triangular matrix $T_{qq}$ with diagonal $(\rival_1,\dots,\rival_q)$, the eigenvector $\yrivec_i$ associated with $\rival_i$ has only its first $i$ components nonzero. Thus $b\tsp \yrivec_i = \sum_{j=1}^{i} b_j \yrivec_{i,j}$, yielding $\abs{b\tsp \yrivec_i} \leq \eta \sqrt{i}$ for $i = 1,\dots,q$, instead of $\abs{b\tsp \yrivec_i} \leq \eta \sqrt{q}$ for all $i$. This matters only if a very large number $\sa$ of eigenpairs is sought, since $q$ can be large.
\end{remark}

We now consider the backward stability of zeroing components of $b$ when they are less than a threshold $\eta$ during rKS. We show that this is equivalent to continuing the \rKS{} algorithm on a slightly perturbed operator $\tilde{A} = A + E$. Starting from \cref{eq:AUqU}:
\begin{align*}
    \A \U & = \U T + \ulast (0 \; \tilde{b}\tsp) +   \ulast (b\tsp \;0) \\
    \A \U -  \ulast (b\tsp \;0)  & = \U T + \ulast (0 \; \tilde{b}\tsp) \\
    \A \U -  \ulast (b\tsp \;0) (\Om \U)^T \Om \U  & = \U T + \ulast (0 \; \tilde{b}\tsp) \; \text{ using $(\Om \U)^T \Om \U = I_\sa$} \\
    (\A -  \ulast (b\tsp \;0) (\Om \U)^T \Om) \U & = \U T + \ulast (0 \; \tilde{b}\tsp) \\
    (\A -  \ulast (b\tsp \;0) 
    \begin{pmatrix}
        (\Om U_q)^T \\
        (\Om \tilde{U})^T
    \end{pmatrix}
    \Om) \U & = \U T + \ulast (0 \; \tilde{b}\tsp) \\
     (\A -  \ulast b\tsp (\Om U_q)^T \Om) \U & = \U T + \ulast (0 \; \tilde{b}\tsp), 
\end{align*}
that is,
\begin{equation}
    \label{eq:perturbed_problem}
    (\A +  E) \U  = \U T + \ulast (0 \; \tilde{b}\tsp) \; \text{ where } \; E \coloneqq -   \ulast b\tsp (\Om U_q)^T \Om.
\end{equation}
We now address the question of the norm of the error $E$ as a perturbation of $\A$. We have, in both Frobenius and spectral norms,
\begin{align}
    \label{eq:errorE}
    \norm{E} & = \norm{ \ulast b\tsp (\Om U_q)^T \Om} \leq \norm{\ulast b\tsp} \norm{(\Om U_q)^T} \norm{\Om} \leq \norm{\ulast} \norm{b\tsp}  \norm{\Om}
\end{align}
since $\ulast b\tsp$ is a rank-1 matrix and $\Om U_q$ is unitary by construction. We know that $\norm{\ulast} \leq 1 / \sqrt{1-\epsilon}$ and $\norm{b\tsp} \leq \eta \sqrt{q}$. However, the factor $\norm{\Om}$ cannot be tightly bounded in general. If the entries of $\Om$ are standard normal, then by \cite[Proposition 10.1]{HalkoFindingStructureRandomness2011}, $\norm{\Om} \leq \sqrt{n} + \sqrt{d}$, where $d$ is the sketching size. To efficiently bound $\norm{\Om}$, recall that we are working within an iterative Krylov procedure, applying $\A$ (or $\tilde{\A}$) only to a small subspace $\K_\sa \subset \Rn{n}$. This is highlighted in any \rKS{} factorization, notably $ \tilde{A} \U  = \U T + \ulast (0 \; \tilde{b}\tsp) $, where $\tilde{A} $ is not considered on its own but applied to $\U$ spanning $\K_\sa$. In other words, the two equations 
\begin{equation}
    \tilde{A} \U  = \U T + \ulast (0 \; \tilde{b}\tsp) \; \text{ and } \; \tilde{A}_{|\K_\sa} \U  = \U T + \ulast (0 \; \tilde{b}\tsp)
\end{equation}
are equivalent, where $\tilde{A}_{|\K_\sa}$ is the restriction of $\tilde{A}$ to $\K_\sa$. In this context, it is relevant to consider the error $E$ restricted to $\K_\sa$, i.e.\@, $E_{|\K_\sa} = - \ulast b\tsp (\Om U_q)^T \Om_{|\K_\sa}$. For any $x \in \K_\sa$ with $\norm{x} = 1$, $\norm{\Om x} \leq \sqrt{1+\epsilon}$ by the $\epsilon$-embedding property. Thus,
\begin{equation}
    \label{eq:tight_error_OmK}
    \norm{\Om_{|\K_\sa}}_2 = \max_{x \in \K_\sa, \norm{x} = 1} \norm{\Om x} \leq \sqrt{1+\epsilon},
\end{equation}
which is a much tighter bound in this situation than $\sqrt{n} + \sqrt{\dsk}$.

We summarize the above discussion on practical subspace deflation in the following theorem.
\begin{theorem}[Backward error of the subspace deflation]
    Assume a partitioned \rKS{} factorization 
    $
    \A (U_q \; \tilde{U}) = (U_q \; \tilde{U})
    \begin{pmatrix}
        T_{qq} & * \\
        0 & \tilde{T}
    \end{pmatrix}
    + \ulast (b\tsp \; \tilde{b}\tsp).$ Then the subspace $U_q$ can be deflated by considering $b$ to be $0$. If each component of $b$ is smaller than a threshold $\eta$ in absolute value, then any Ritz pair $(\rival, \rivec = U_q \yrivec)$ coming from an eigenpair $(\rival,\yrivec)$ of $T_{qq}$ satisfies the error bound
    \begin{equation}
        \norm{\A \rivec - \rival \rivec} \leq \frac{\sqrt{q}}{\sqrt{1-\epsilon}} \eta.
    \end{equation}
    Moreover, there exists a set of unit norm vectors $Z \in \Rnm{n}{q}$ with $\Span{Z} = \Span{U_q}$ and a diagonal matrix $D \in \Rnm{q}{q}$ such that
    \begin{equation}
        \enorm{\A Z - Z D} \leq q \sqrt{\frac{1+\epsilon}{1-\epsilon}} \eta \quad \text{and} \quad \fnorm{\A Z - Z D} \leq q^{3/2} \sqrt{\frac{1+\epsilon}{1-\epsilon}} \eta.
    \end{equation}
    Besides, the \rKS{} method can be continued on the smaller, perturbed problem 
    \begin{equation}
        (I - P_{U_q}^\Om) (A + E)\tilde{U} = \tilde{U} \tilde{T} + \ulast \tilde{b}\tsp \; \text{ where } \; \norm{E_{|\K}}_{F,2} \leq \sqrt{q} \sqrt{\frac{1+\epsilon}{1-\epsilon}} \eta 
    \end{equation}
    which satisfies the properties of the exact subspace deflation from \cref{th:exactSubspaceDeflation}.
\end{theorem}
\begin{proof}
    The first bound on a Ritz pair comes from \cref{eq:etaControlRitz}.
    To prove the existence of $Z$ and $D$, write the eigendecomposition of $T_{qq}$ as $T_{qq} W = W D$ with $D$ diagonal. Define $X \coloneqq U_q W$, so
    $\A X - X D = \ulast b\tsp W.$
    To normalize the columns $x_i$ of $X$, right-multiply by the scaling diagonal matrix $\Sigma^{-1} = \diag(\frac{1}{\norm{x_1}}, \dots, \frac{1}{\norm{x_q}})$ and set $Z \coloneqq X \Sigma^{-1}$. Note that $D$ and $\Sigma^{-1}$ commute as diagonal matrices. Thus,
    $$\A Z - Z D = \ulast b\tsp W \Sigma^{-1}.$$
    To bound the right-hand side, use the $\epsilon$-embedding property \cref{eq:epsembedd} to get $\sqrt{1-\epsilon} \leq 1 / \norm{x_i} \leq \sqrt{1+\epsilon}$ since $\norm{\Om x_i} = \norm{\Om U_q w_i} = 1$ (as $\Om U_q$ is unitary and $w_i$ is unit norm). Therefore, $\enorm{\Sigma^{-1}} \leq \sqrt{1+\epsilon}$ and $\fnorm{\Sigma^{-1}} \leq \sqrt{q}  \sqrt{1+\epsilon} $. Since $W$ is a set of $q$ unit norm eigenvectors, $\fnorm{W} = \sqrt{q}$ and $\enorm{W} \leq \sqrt{q}$. Using $\norm{\ulast b\tsp} \leq \frac{\sqrt{q}}{\sqrt{1-\epsilon}} \eta$ from \cref{eq:etaControlRitz}, for $\A Z - Z D = \ulast b\tsp W \Sigma^{-1}$:
    $$\enorm{\A Z - Z D} \leq q \sqrt{\frac{1+\epsilon}{1-\epsilon}} \eta, \quad \fnorm{\A Z - Z D} \leq q^{3/2} \sqrt{\frac{1+\epsilon}{1-\epsilon}} \eta.$$
    The existence of the perturbed problem follows from the discussion above. In the theorem, we add the projection $ (I - P_{U_q}^\Om)$ to restrict \cref{eq:perturbed_problem} to $\tilde{U}$. The bound on $\norm{E_{|\K}}_{F,2}$ comes from \cref{eq:errorE}, $\norm{\ulast b\tsp} \leq \frac{\sqrt{q}}{\sqrt{1-\epsilon}} \eta$, and \cref{eq:tight_error_OmK}.
\end{proof}

\section{Numerical experiments}
\label{sec:numericals}

In this section, we demonstrate the numerical efficiency of the \rKS{} algorithm. All experiments are performed using Julia version 1.10 \cite{Bezanson2017Juliafreshapproach}, running on a computational node with 2x Cascade Lake Intel Xeon 5218 CPUs (16 cores, 2.4GHz) and 192GB RAM. For randomized embeddings, we use sparse-sign matrices with sketching size $\dsk = 2m$, where $m$ is the Krylov dimension, and each column of $\Om \in \Rnm{\dsk}{n}$ contains $\zeta = 8$ normalized random signs $\frac{\pm 1}{\sqrt{\dsk}}$. Sparse-sign embeddings are described in \cite{MartinssonRandomizednumericallinear2020}.

We compare the \rKS{} algorithm to two deterministic eigensolvers: Implicitly Restarted Arnoldi (IRA) and Krylov-Schur (KS). IRA is Julia's native sparse eigensolver, called via \emph{eigs}, and serves as our reference. KS is our own implementation, essentially \cref{alg:rKS} without sketch orthogonalization, using classical Gram-Schmidt with reorthogonalization (CGS2) for stability. Using CGS or modified Gram-Schmidt (MGS) for deterministic KS resulted in poor performance due to loss of orthogonality in the Krylov basis. For rKS, we use Randomized Gram-Schmidt (RGS) \cite{BalabanovRandomizedGramSchmidt2022} to sketch-orthogonalize the Krylov basis. RGS, with stability similar to MGS, was sufficient in all our tests. As noted in \cref{sec:prelim}, RGS requires half the flops of CGS, and thus a quarter of the flops of CGS2. However, this rough flop count does not include computational constants in the small dimension $m$ (e.g., $O(k^3)$), among other costs. In practice, we expect a speedup factor of 2 to 4, meaning rKS should be 2 to 4 times faster than IRA/KS in our sequential setup.

\subsection{Computational speedup of randomized Krylov Schur with increasing input dimension}
This subsection investigates how rKS scales as the input size $n$ increases for $\A \in \Rnm{n}{n}$. We seek $\sa = 40$ eigenvalues of either largest modulus (LM) or smallest modulus (SM), using a Krylov dimension $m = 2k = 80$. The input matrices are synthetic, non-symmetric, and tri-diagonal. Diagonal entries are specified in \cref{tab:SynthTestMatrices}, while sub-diagonal entries are Gaussian:
\begin{equation}
    \label{eq:synthMathSubdiags}
    \quad \A_{i+1,i} = \frac{g_{i+}}{100}, \; \A_{i-1,i} = \frac{g_{i-}}{100},
\end{equation}
where $g_{i \pm}$ are drawn from $\mathcal{N}(0,1)$. These Gaussian sub-diagonals act as noise for the main spectrum on the diagonal. The input dimension $n$ increases from $n_{min} = 10^5$ to $n_{max} = 5 \times 10^6$.

\begin{table}
        \caption{Test matrices for \cref{fig:synth_timings,fig:synth_iterations}. Here, $i \in \mathcal{I}$ where $\mathcal{I}$ is the set of $n$ equispaced points from $2$ to $10$. These matrices are tridiagonal, with non-diagonal elements being Gaussian noise specified in \cref{eq:synthMathSubdiags}.}
        \label{tab:SynthTestMatrices}
    \begin{center}
    \begin{tabular}{|c|c|c|c|c|}
        \hline
        \rule{0pt}{3ex}  \textbf{Type of spectrum} & Exponential &  Logarithmic &  Harmonic roots & Geometric decay   \\[1ex]
        \hline
        \rule{0pt}{3ex}  \textbf{Diagonal entries} & $\exp(i/10)$ & $\log(i +1)$  & $1 + 1/i^2$ & $0.99^i$ \\[1ex]
        \hline
    \end{tabular}
\end{center}
\end{table}

The execution times for eight different configurations are shown in \cref{fig:synth_timings}, and the corresponding number of restart iterations for convergence is shown in \cref{fig:synth_iterations}. The last experiment, "Exponential SM, 10 iterations" (bottom right), is a special case where the number of iterations is fixed at 10, which is insufficient for convergence. It serves as a reference for the speedup provided by rKS for the same number of iterations across all methods. In other scenarios, all methods converge, but typically in a different number of iterations. Convergence is declared when all residual errors $\abs{\blast\tsp \yrivec_i} / \abs{\rival_i}$ from \cref{eq:rks_reserror_bound} are below the threshold $\eta = 10^{-10}$ for each pair $(\yrivec_i, \rival_i)$, $i = 1,\dots,\sa$. Discrepancies in iteration counts may be due to the artificial structure of these matrices. In the next subsection (\cref{sec:scaling_krylov}) on real-world examples, this effect is reduced, as seen in \cref{fig:mosaic_iterations}.

We observe two main properties of rKS. First, it is always the fastest method among the three eigensolvers. Since \cref{fig:synth_timings} uses a logarithmic scale, the fact that rKS is consistently below IRA/KS indicates a multiplicative speedup, ranging from 2 to 3, consistent with our expectations. This speedup is clearly due to the use of randomized orthogonalization, as shown by the fixed-iteration experiment, where rKS outperforms the other two. Second, rKS often requires fewer iterations to converge than IRA or KS. It remains an open question to understand whether randomization enhances stability for the eigenvalue problem, in the sense of reducing the number of iterations required for convergence. 

\Cref{fig:res_vp} addresses the quality of the Ritz eigenvalues obtained as approximations to the true eigenvalues. It shows the decrease in residual error over iterations for KS and rKS for runs with $n = 5 \times 10^6$ from \cref{fig:synth_timings}, as well as the real parts of the computed eigenvalues. The maximum residual error among the sought eigenpairs is shown at each iteration, i.e.\@, the maximum of $\abs{\blast\tsp \yrivec_i} / \abs{\rival_i}$ from \cref{eq:rks_reserror_bound} over $i = 1,\dots,\sa$. We focus on the real parts, since the eigenvalues are mainly real, with small imaginary parts due to Gaussian noise on the subdiagonals. KS and rKS perform similarly in terms of error, with rKS sometimes reaching tolerance a few iterations earlier.  We note that our randomized method tends to introduce convergence spikes. However, we emphasize that, on our test matrices, rKS is never less accurate than its deterministic counterpart. The eigenvalues found are the same for all three methods, which is an advantage for rKS since it computes them faster. Importantly, randomization does not reduce accuracy in the target eigenvalues, as discussed at the end of \cref{sec:KSasRR}.

\begin{figure}
    \centering
    \includegraphics{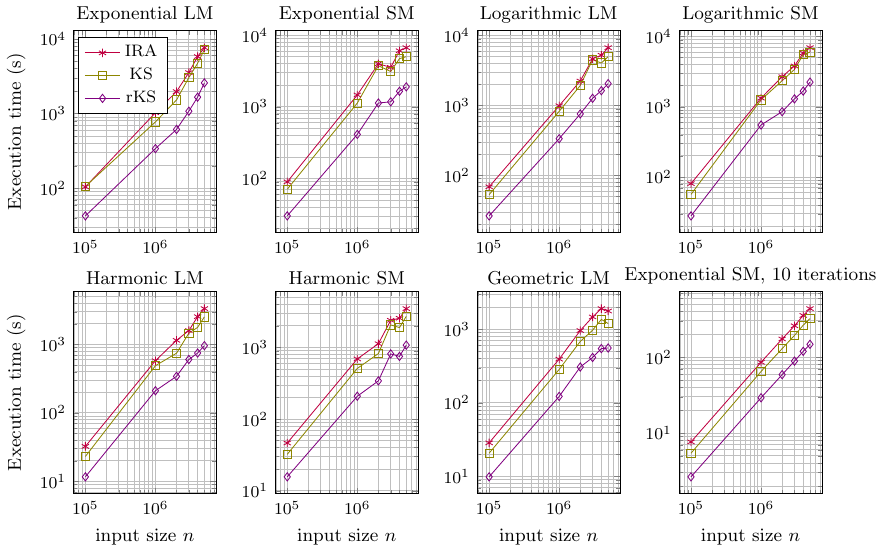}
    \caption{Execution time of IRA, KS and rKS on synthetic matrices from \cref{tab:SynthTestMatrices}. For each data point except on "Exponential SM, 10 iterations", all methods have converged and found $\sa = 40$ eigenpairs of either largest modulus (LM) or smallest modulus (SM) in a Krylov dimension $m = 2\sa = 80$. The input size $n$ increases from $10^5$ to $5.10^6$. Speedup is always obtained by using rKS, up to 3x faster. On "Exponential SM, 10 iterations", all methods performed 10 iterations at each data point, without convergence. It serves as a benchmark for speedup thanks to randomization, all things considered.}
    \label{fig:synth_timings}
\end{figure}
\begin{figure}
    \centering
    \includegraphics{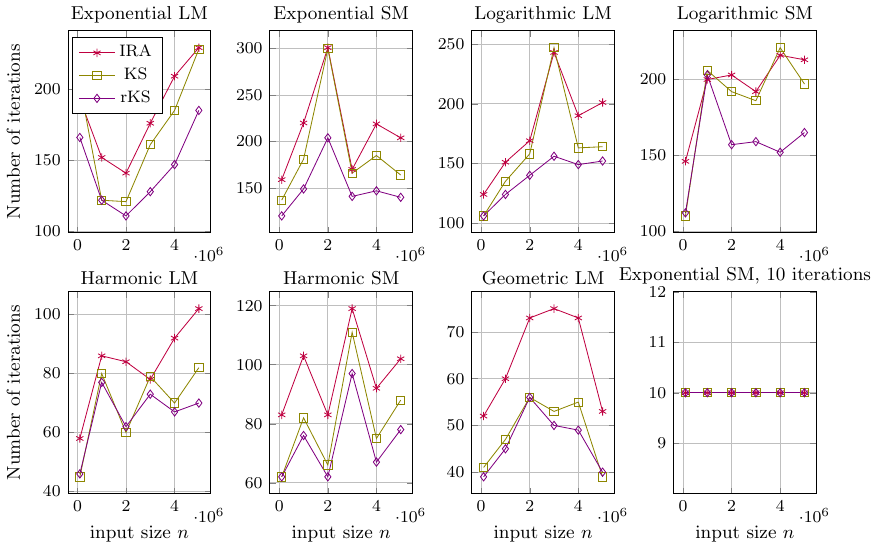}
    \caption{Corresponding number of iterations from experiments of \cref{fig:synth_timings} for each method to converge and find $\sa = 40$ eigenpairs of either largest modulus (LM) or smallest modulus (SM) in a Krylov dimension $m = 2\sa = 80$. The discrepancy in these numbers of iterations might be a consequence of the synthetic structure of the inputs matrices, as it is not observed for real case matrices in \cref{fig:mosaic_iterations}. However, rKS always perform the less number of iterations, a possible consequence of stability induced by randomization.}
    \label{fig:synth_iterations}
\end{figure}
\begin{figure}
    \centering
    \includegraphics{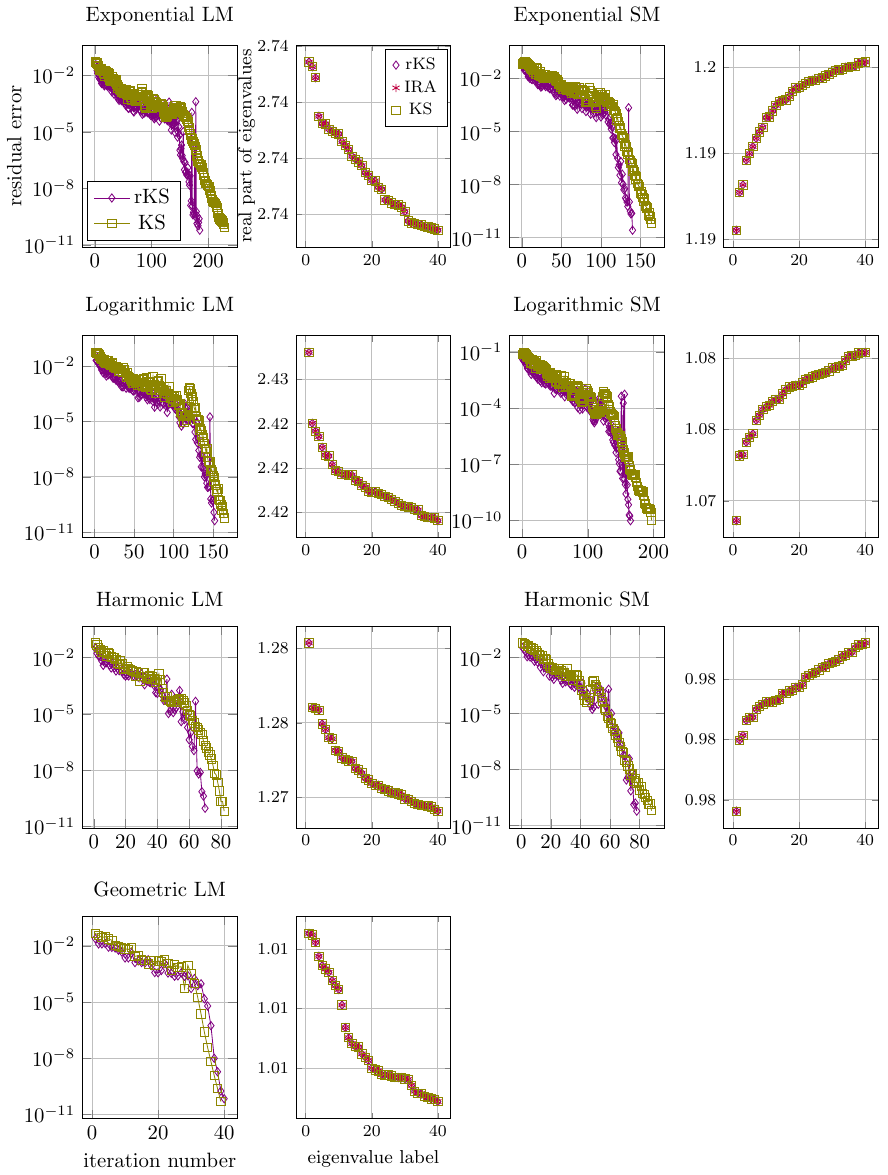}
    \caption{Maximum residual errors (left) alongside real parts of the eigenvalues computed (right), for the eight experiments where $n = 5.10^6$ from \cref{fig:synth_timings} using matrices from \cref{tab:SynthTestMatrices}. The residual errors are computed with the term $\abs{\blast\tsp \yrivec} / \abs{\rival}$ from \cref{eq:rks_reserror_bound}. As such, they are equal to $\frac{\norm{\A \rivec - \rival \rivec}}{\norm{\A \rivec}}$ or close to it up to a factor of $\bigO(\eps)$, for KS and rKS, respectively. They are not available for the IRA function. On the right of each residual errors plots, the real parts of the eigenvalues found are displayed to show that all three methods find the same eigenvalues. The imaginary parts were usually small and do not bring information, since the syntetic matrices have large diagonal values compared to their subdiagonal Gaussian noise. Randomization brings spikes in the errors but also tends to converge earlier.}
    \label{fig:res_vp}
\end{figure}
\begin{remark}
    In these experiments, we do not display the Geometric SM case in order to show the fixed number of iterations scenario, but we add here that the results for it were similar to all the previous configurations. 
\end{remark}

\subsection{Acceleration benefits of randomized Krylov Schur for varying Krylov dimension}
\label{sec:scaling_krylov}
In the next experiment, we compare rKS, IRA, and KS as the Krylov dimension $m$ increases for a fixed number of target eigenpairs $\sa$. We test the methods on matrices from various applications, taken from the SuiteSparse Matrix Collection \cite{Davis2011universityFloridasparse}; see \cref{tab:testMatrices} for details.

We seek $\sa = 50$ eigenpairs, varying the Krylov dimension from $m_{min}=75$ to $m_{max}=200$. Most cases target eigenpairs of largest modulus (LM), with one instance of smallest modulus (SM). The convergence threshold is set to $\eta = 10^{-10}$, and a maximum of 300 outer iterations (restarts) is allowed. The exception is the bottom right experiment, "Stokes SM, 10 iterations", where each method is limited to 10 restarts, serving as a benchmark for scaling in $m$. Results are shown in \cref{fig:mosaic_timings} (execution time) and \cref{fig:mosaic_iterations} (restart iterations to convergence). Methods are considered converged if the number of restarts is less than 300, which is almost always the case, except for a few instances at $m = 75$. Recall that $m = \ba$, so $\sa = 50$ and $m = 25$ leaves $p= 25$. Since $p$ is the number of shifts (discarded directions) at each restart, it must be large enough for fast convergence.

From \cref{fig:mosaic_timings}, rKS is clearly the fastest in all experiments. The speedup varies with problem size and Krylov dimension, but is typically 2 to 4 times faster, as expected. Moreover, rKS is less sensitive to the ratio $m/\sa$, often set to 2 by empirically choosing $m = 2\sa$. This property arises from sketch orthogonalization, which mitigates the challenge of selecting the total Krylov dimension $m$ relative to the target $\sa$; see \cite{ShahzadehFazeli2015keychoosesubspace} for discussion in the deterministic setting. The impact of $m/\sa$ on iteration count is best seen in \cref{fig:mosaic_iterations}, which highlights the trade-off between a small Krylov subspace (reducing orthogonalization cost) and a large one (reducing restarts). All three methods perform a similar number of iterations, especially for $m \geq 2\sa$, and the discrepancy in iteration counts for synthetic matrices in \cref{fig:synth_iterations} should not be generalized. However, rKS always performs the fewest iterations when there is a difference, reinforcing observations from previous experiments.

Finally, the bottom right experiment "Stokes SM, 10 iterations" shows the speedup obtained by rKS when the number of restarts is fixed but $m$ increases. This plot is not in logarithmic scale, to contrast with \cref{fig:synth_timings} and provide a different visualization. The speedup is still about 2 to 3 times faster than the deterministic methods.

\begin{table}
        \caption{Test matrices for \cref{fig:mosaic_timings} from the Suite Sparse Matrix Collection. They are all non symmetric except for Ga41As41H72.}
        \label{tab:testMatrices}
    \begin{center}
    \begin{tabular}{|l|l|l|l|}
        \hline
        \textbf{Name} & \textbf{Size n} & \textbf{Nonzeros} & \textbf{Problem} \\
        \hline
        Ga41As41H72 & $268,096$ & $18,488,476$ & Theoretical/Quantum Chemistry Problem \\
        Atmosmodl & $1,489,752$ & $10,319,760$ & Computational Fluid Dynamics \\
        ML\_Geer & $1,504,002$ & $110,686,677$ & Structural Problem  \\
        SS & $1,652,680$ & $34,753,577$ & 	Semiconductor Process Problem \\
        Vas\_stokes\_4M & $	4,382,246$ & $131,577,616$ & Semiconductor Process Problem \\
        Stokes & $11,449,533$ & $349,321,980$ & Semiconductor Process Problem
        \\
        \hline
    \end{tabular}
\end{center}
\end{table}

\begin{figure}
    \centering
    \includegraphics{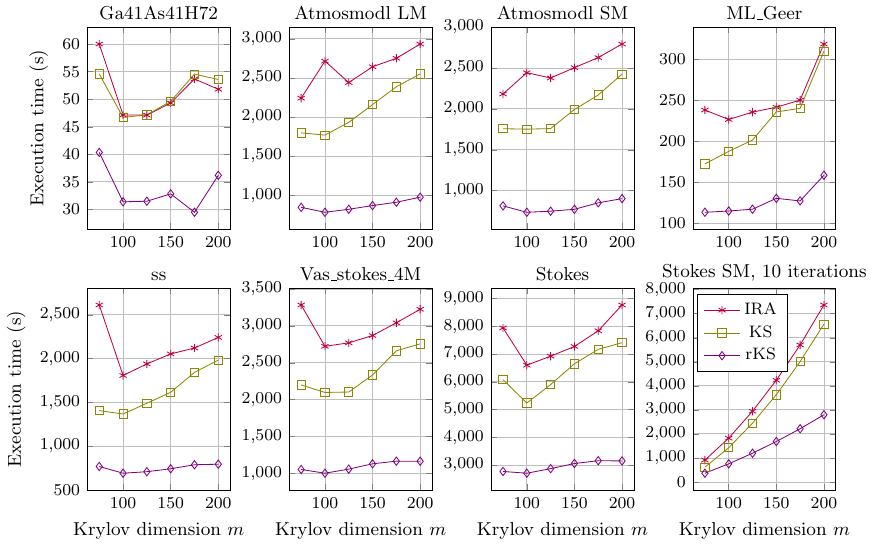}
    \caption{Execution time to find $\sa = 50$ eigenpairs of different matrices summarized in \cref{tab:testMatrices} for an increasing Krylov dimension $m$, from $1.5\sa = 75$ to $4\sa = 200$. The number of iterations taken for each method are reported in \cref{fig:mosaic_iterations}. For all data points, rKS is the fastest method, being in the range of twice as fast as IRA/KS. The rKS method seems also to be less sensitive to the ratio $m/\sa$, that is, it converges in the same amount of time regardless of $m$ compared to a fixed $\sa$. Bottom right plot is a special case where all methods perform 10 iterations.}
    \label{fig:mosaic_timings}
\end{figure}
\begin{figure}
    \centering
    \includegraphics{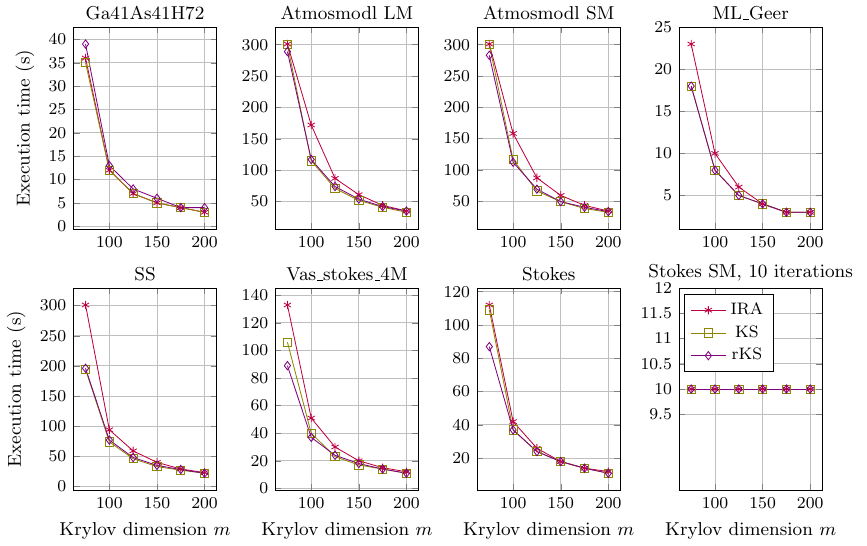}
    \caption{Number of iterations to find $\sa = 50$ eigenpairs of different matrices summarized in \cref{tab:testMatrices} for an increasing Krylov dimension $m$, from $1.5\sa = 75$ to $4\sa = 200$. The execution times for each method are reported in \cref{fig:mosaic_iterations}. All the methods roughly perform the same number of iterations, even though rKS always does the least number of them.}
    \label{fig:mosaic_iterations}
\end{figure}

\section{Conclusion}
In this work, we develop an eigensolver for computing a small number of eigenpairs from large, sparse matrices. The method, called randomized Krylov-Schur, is adapted from the Krylov-Schur algorithm \cite{Stewart2002KrylovSchurAlgorithm}. It leverages sketch orthogonalization techniques, and we showed how the resulting sketch-orthogonal factorizations fit within the framework of Krylov (Schur) decompositions. We demonstrated that sketch-orthonormal Krylov decompositions are equivalent to \rArno{} factorizations \cite{Damas2024RandomizedImplicitlyRestarted,Balabanov2021RandomizedblockGram}, establishing a connection between rKS and randomized Implicitly Restarted Arnoldi (rIRA). As an extension of rIRA, rKS relaxes the structural constraints of randomized Arnoldi factorization, resulting in a simpler implementation while maintaining, or even improving, performance due to stable Schur reordering compared to the less stable shifted QR algorithm.

This framework of sketch-orthonormal Krylov decomposition enables efficient deflation procedures, and we analyzed how these can be implemented and what errors may arise in practice. Numerical experiments confirm that rKS is faster than deterministic eigensolvers, while preserving accuracy.

Future work includes exploring mixed-precision arithmetic within sketch orthogonalization to further reduce computational cost in the large dimension $n$, while maintaining double precision in the sketch dimension. Developing a finite precision analysis of the method would also be a valuable step toward integrating rKS into standard libraries.

\section*{Acknowledgments}
This project has received funding from the European Research
Council (ERC) under the European Union’s Horizon 2020 research and innovation program (grant agreement No 810367).

\bibliographystyle{alpha}
\bibliography{jabref}

\end{document}